\documentclass[12pt,reqno]{amsart}
\usepackage{hyperref}
\usepackage{cleveref}
\usepackage{amsmath,amssymb}
\usepackage[english]{babel}
\usepackage{caption}
\usepackage{amssymb,amsmath,euscript,enumerate,tikz}
\usepackage[margin=1in]{geometry}
\usepackage{graphicx}
\usepackage{float}
\usepackage{pgf,tikz}
\usepackage{mathrsfs}
\usepackage{upgreek}
\usepackage[normalem]{ulem}

\newcommand \reg{\operatorname{reg}}

\newcommand \Tor{\operatorname{Tor}}

\newcommand \pd{\operatorname{pdim}}
\newcommand \iv{\operatorname{iv}}

\newcommand \depth{\operatorname{depth}}

\newcommand \K{\mathbb{K}}

\newcommand{\dist}{\operatorname{dist}}

\newtheorem{theorem}{Theorem}[section]
\newtheorem{definition}[theorem]{Definition}
\newtheorem{lemma}[theorem]{Lemma}
\newtheorem{proposition}[theorem]{Proposition}
\newtheorem{example}[theorem]{Example}

\newtheorem{question}[theorem]{Question}

\newtheorem{remark}[theorem]{Remark}

\begin{document}
\title[Depth of Binomial Edge Ideals in terms of Diameter and Vertex Connectivity]{Depth of Binomial Edge Ideals in terms of Diameter and Vertex Connectivity}
\author[A. V. Jayanthan]{A. V. Jayanthan}
\email{jayanav@iitm.ac.in}
\address{Department of Mathematics, Indian Institute of Technology Madras, Chennai-600036, India}
\author[Rajib Sarkar]{Rajib Sarkar}
\email{rajib.sarkar63@gmail.com}
\address{Stat-Math Unit, Indian Statistical Institute, 203 B.T. Road, Kolkata--700108, India}

\begin{abstract}
Let $G$ be a simple connected non-complete graph and $J_G$ be its binomial edge ideal in a polynomial ring $S$. Using certain invariants associated to graphs, say $U(G)$, Banerjee and N\'{u}\~{n}ez-Betancourt gave an upper bound for the depth of $S/J_G$, and Rouzbahani Malayeri, Saeedi Madani and Kiani obtained a lower bound, say $L(G)$. Hibi and Saeedi Madani gave a structural classification of  graphs satisfying $L(G)=U(G)$.  In this article, we give structural classification of graphs satisfying $L(G)+1=U(G)$. We also compute the depth of $S/J_G$ for all such graphs $G$.
\end{abstract}

\dedicatory{Dedicated to J\"urgen Herzog on the occasion of his $80^{th}$ birthday}
\maketitle
\section{Introduction}

Let $R = \K[x_1,\ldots,x_m]$ be the standard graded polynomial ring over an arbitrary
field $\K$ and $M$ be a finitely generated graded  $R$-module. 
Let
\[
0 \longrightarrow \bigoplus_{j \in \mathbb{Z}} R(-j)^{\beta_{p,j}(M)} 
{\longrightarrow} \cdots {\longrightarrow} \bigoplus_{j \in \mathbb{Z}} R(-j)^{\beta_{0,j}(M)} 
{\longrightarrow} M\longrightarrow 0,
\]
be the minimal graded free resolution of $M$. The number $\beta_{i,j}(M)$ is called the 
$(i,j)$-th \textit{graded Betti number} of $M$. One can read off a lot of homological information about the module from the minimal graded free resolution. For example, the projective dimension, which is denoted by $\pd(M)$ and defined as $$\pd(M):= \max\{i : \beta_{i,j}(M) \neq 0 \text{ for some } j\}.$$
It may be noted that in this case, computing the projective dimension is equivalent to computing the depth of module, thanks to Auslander-Buchsbaum formula. Another important homological invariant that one gets from the minimal graded free resolution is called Castelnuovo-Mumford regularity (or simply, regularity) of M, denoted by reg(M), and  is defined as
\[
\reg(M)=\max\{j-i : \beta_{i,j}(M) \neq 0 \text{ for some } i \geq 0 \}.
\]


Associating an ideal to a geometric/topological/combinatorial object and understanding its properties through the algebraic properties of the ideal has been a classical theme of research. When $G$ is a finite simple graph, Herzog et al. \cite{HH1} and independently Ohtani \cite{oh}, introduced binomial edge ideal associated with the graph. Let $G$ be a simple graph on $[n] = \{1, \ldots, n\}$. The binomial edge ideal corresponding to $G$ is the ideal $J_G$ generated by $\{x_iy_j - x_jy_i ~:~ \{i, j\} \in E(G)\}$ in the polynomial ring $S=\mathbb{K}[x_1, \ldots, x_n, y_1, \ldots, y_n]$. The binomial edge ideals of $G$ can also be thought of as a natural generalization of the well-studied determinantal ideal of the generic matrix 
\[\begin{bmatrix}
	x_1 \ x_2 \cdots x_n \\
	y_1 \ y_2 \cdots y_n
\end{bmatrix}
.\]
On the other hand, binomial edge ideals arise naturally
in Algebraic Statistics, in the context of conditional independence ideals, see \cite[Section 4]{HH1}.

The connection between combinatorial invariants associated to $G$ and algebraic invariants associated to $J_G$ has been of intense study in the past one decade (see \cite{BN17,dav,dav2,her1,Hibi-Madani,JAR2,AR2,MKM-depth,MKM-JA,MKM-Conj} for a partial list and see \cite{Madani-survey} for a survey on this topic). Cohen-Macaulayness is an  algebraic property of a module that is of interest to algebraists as well as researchers in other areas such as geometers and topologists. In the context of binomial edge ideals, there have been attempts to understand the Cohen-Macaulayness of $S/J_G$ in terms of the combinatoics of the graph $G$, \cite{dav,dav2,her1,Rinaldo-Rauf,Rinaldo-Cactus}. To understand the Cohen-Macaulayness, it is necessary to understand the depth of the module. 
%
%
It is known that if $G$ is a complete graph on $n$ vertices, then $S/J_G$ is Cohen-Macaulay with $\depth (S/J_G) =n+1$, (see \cite[Corollary 2.8]{BV88}). Ene, Herzog and Hibi first studied the depth of binomial edge ideal and they proved that $\depth(S/J_G)=n+1$ whenever $G$ is a connected block graph on $n$ vertices, \cite{her1}. 
For a connected non-complete graph, a general combinatorial upper bound for the depth of $S/J_G$ was first obtained by Banerjee and N\'{u}\~{n}ez-Betancourt in \cite{BN17} and they proved that for a connected non-complete graph $G$ on $[n]$, 
$$\depth(S/J_G) \leq n+2-\kappa(G),$$  where $\kappa(G) = \min\{|T| ~:~ T \subset [n] \text{ and the induced graph on } [n]\setminus T \text{ is disconnected}\}$. This invariant is called the \textit{vertex connectivity} of $G$.  For $i, j \in [n],$ let $\dist_G(i,j)$ denote the number of edges in a shortest path connecting $i$ and $j$ in $G$. Let $d(G) = \max\{\dist_G(i,j) ~:~ i, j \in [n]\}$ denote the \textit{diameter} of $G$. In \cite{MKM-depth}, Rouzbahani Malayeri, Saeedi Madani and Kiani proved that 
$$d(G) + f(G) \leq \depth (S/J_G),$$ where $f(G)$ denotes the number of simplicial vertices of $G$. Then the natural question comes that whether there exists a graph $G$ such that $d(G)+f(G)=n+2-\kappa(G)$. Recently, Hibi and Saeedi Madani showed the existence of such graphs and characterized connected non-complete graphs with such property, \cite{Hibi-Madani}. In this article, we consider the next case and characterize connected non-complete graphs $G$ with the property that $$d(G)+f(G)+1=n+2-\kappa(G).$$ 

In Section 3, we first give necessary and sufficient conditions for a graph $G$ to satisfy the property $d(G)+f(G)+1 = n+2-\kappa(G)$, (\Cref{charac-thm}). We then completely classify these graphs in terms of the structural properties: \\[2ex]
\textbf{Theorem}
\ref{induced-cycle-charac}. 
{\em Let $G$ be a graph on $[n]$ such that $d(G)+f(G)+1=n+2-\kappa(G).$ 
\begin{enumerate}
    \item If $\kappa(G)=1$, then either $G$ is chordal or $G$ has precisely one induced $C_4$ and has no induced $C_l$ for $l \geq 5$. 
    \item If $\kappa(G)\geq 2$ and $d(G)=2$, then $G$ is a chordal graph.
    \item If $\kappa(G)=2$ and $d(G)=3$, then either $G$ is  chordal or $G$ has precisely one induced $C_4$ and has no induced $C_l$ for $l \geq 5$.
\end{enumerate}
}

For such a graph $G$,  $d(G)+f(G)\leq \depth(S/J_G)\leq d(G)+f(G)+1$. In Section 4, we describe completely the graphs that attain the lower bound or upper bound. We do this by splitting the study mainly into two cases, namely $\kappa(G)=1$ and $\kappa(G) \geq 2$. When $\kappa(G)=1$, we first consider chordal graphs in this category. We define a graph $H$ having $d+2$ vertices and $d+3$ edges, see Figure \ref{figure-H}. We show that if $G$ is chordal satisfying $d(G)+f(G)+1 =  n+1$ ($\kappa(G)=1)$ and not  having $H$ as an induced subgraph, then $\depth(S/J_G) = n+1$, \Cref{claim1}. For chordal graphs with $\kappa(G)=1$ and having $H$ as an induced subgraph, we show that $\depth(S/J_G) = n$, \Cref{claim2}. When $\kappa(G)=1$ and $G$ non-chordal, then we describe the depth of these graphs using certain combinatorial properties, Theorems \ref{claim3}, \ref{claim4}.
In Theorems \ref{claim5} - \ref{claim8}, we obtain the depths of these graphs with $\kappa(G)\geq 2$. We again do this by splitting the study into chordal and non-chordal classes.
Whenever one gets a bound on an invariant, it is interesting to ask if one can characterize all the objects satisfying the lower bound. In our case, this seems like a very hard question. The  Hibi-Madani class is one which attains the lower bound. In Section 5, we obtain certain chordal graphs which having depth strictly bigger than the minimal value, (\Cref{coc}). For unicyclic and quasi-cycle graphs, we obtain necessary conditions for the depth to attain the minimum value, (\Cref{nec-min-depth}).

\section{Preliminaries}
We recall some basic notation, definitions and some known results which will be used throughout this article.

Let $G$ be a simple graph with the vertex set $V(G)$ and the edge set $E(G)$. A subgraph $H$ of $G$ is said to be an \textit{induced subgraph} if for $u, v \in V(H)$, $\{u,v\} \in E(H)$ if and only if $\{u,v\} \in E(G)$. If $A \subseteq V(G)$, then $G[A]$ denotes the induced subgraph on the vertex set $A$. For a vertex $v$ of $G$, $G \setminus v$ denotes the induced subgraph on the vertex set $V(G) \setminus \{v\}$. For $e \in E(G)$,  $G\setminus e$ denotes the graph with vertex set $V(G)$ and the edge set $E(G) \setminus \{e\}$. If $v \in V(G)$, then $N_G(v) := \{u ~:~ \{u,v\} \in E(G)\}$ is called the \textit{neighborhood of }$v$. For $v \in V(G)$, let $G_v$  denote the graph with vertices $V(G)$ and $E(G_v) =  E(G) \cup \{\{u, w\} ~:~ u, w \in N_G(v)\}$.  The graph $G$ is called a \textit{chordal graph} if every induced cycle of $G$ has length $3$. We denote by $K_m$, a complete graph on $m$ vertices. We recall the definition of the clique sum of two graphs.

\begin{definition}
Let $G_1$ and $G_2$ be two subgraphs of $G$ such that $G_1\cap G_2\cong K_m$, $V(G)=V(G_1)\cup V(G_2)$ and $E(G)=E(G_1)\cup E(G_2).$ Then $G$ is called the clique sum of $G_1$ and $G_2$ along the complete graph $K_m$ and it is denoted by $G_1\cup_{K_m}G_2$. If $m=1$, i.e., $K_1=\{v\}$, (respectively $m=2$ , i.e., $K_2 = e=\{v_1,v_2\}$), then we denote the clique sum by $G_1 \cup_v G_2$ (respectively $G_1 \cup_e G_2$).
\end{definition}

Let $\Delta$ be a simplicial complex. A facet $F$ of $\Delta$ is called a \textit{leaf} if either it is the only facet of $\Delta$ or there exists a facet  $G$, called a \textit{branch} of $F$, such that for any facet $H(\neq F)$ of $\Delta$, $H \cap F \subseteq G \cap F$. A \textit{simplicial vertex} of $\Delta$ is a vertex that is contained in only one facet. The simplicial complex $\Delta$ is called a \textit{quasi-forest} if its facets can be ordered $F_1, \ldots, F_r$, called a \textit{leaf order}, such that $F_i$ is a leaf of the simplicial complex with facets $F_1, \ldots, F_{i-1}$. 

If a graph $G$ is the clique sum of two subgraphs along a simplicial vertex, then $G$ is said to be a \textit{decomposable graph}. A graph $G$ is called an \textit{indecomposable graph} if it is not a decomposable graph.

For a graph $G$, a clique is a complete subgraph of $G$. The collection of cliques of $G$ forms a simplicial complex, called the \textit{clique complex} of $G$, denoted by $\Delta(G)$. A vertex of $G$ is called a \textit{simplicial vertex} if it is a free vertex of $\Delta(G)$, i.e., a vertex $v$ is a simplicial vertex of $G$ if it is contained in only one facet of $\Delta(G)$. A vertex of $G$ which is not a simplicial vertex is called an \textit{internal vertex} of $G$ and the number of internal vertices is denoted by $\iv(G)$. For $T\subseteq V(G)$, $c_G(T)$ denotes the number of components of $G\setminus T$. A subset $T\subseteq V(G)$ is said to be a cutset if $c_G(T\setminus \{v\})<c_G(T)$ for every vertex $v\in T$. The set of all cutsets of $G$ is denoted by $\mathcal{C}(G)$. If $T=\{v\}\in \mathcal{C}(G)$, then $v$ is called a \textit{cut vertex}. If $H$ is a  subgraph of $G$ and $v \in V(G)$, then $\dist(v,H):=\min\{\dist_G(u,v) : u \in V(H)\}$.

We call the collection of graphs $G$ satisfying the property $d(G)+f(G) = |V(G)|+2-\kappa(G)$, \textit{Hibi-Madani class}. It may be noted that if $G$ belongs to the Hibi-Madani class, then by \cite[Corollary 2.5]{Hibi-Madani}, $\depth(S/J_G) = d(G)+f(G)$.

Let $A=\K[x_1,\dots,x_m]$, $A'=\K[y_1,\dots,y_n]$ and $B=\K[x_1,\dots,x_m,y_1,\dots,y_n]$ be
polynomial rings. Let $I\subseteq A$ and $J\subseteq A'$ be
homogeneous ideals. Then the minimal free resolution of $B/(I+J)B$ can
be obtained by the tensor product of the minimal free resolutions of
$A/I$ and $A'/J$. Therefore, for all $i,j$, we get

\begin{align}\label{Bettiproduct}
\beta_{i,i+j}\left(\frac{B}{(I+J)B}\right) =
\underset{{\substack{i_1+i_2=i \\
j_1+j_2=j}}}{\sum}\beta_{i_1,i_1+j_1}\left(\frac{A}{I}\right)\beta_{i_2,i_2+j_2}\left(\frac{A'}{J}\right).
\end{align}

We recall the following lemma due to Ohtani.
\begin{lemma}$($\cite[Lemma 4.8]{oh}$)$\label{ohtani-lemma}
Let $G$ be a  graph on $V(G)$ and $v\in V(G)$ such that $v$ is an internal vertex. Then $J_G$ can be written as
	$$J_{G}=J_{G_v}\cap ((x_v, y_v) + J_{G\setminus v}).$$
	\end{lemma}
	Therefore, we have the following short exact sequence:
      \begin{align}\label{ohtani-ses}
    0\longrightarrow \frac{S}{J_{G}}\longrightarrow 
    \frac{S}{(x_v,y_v)+J_{G \setminus v}} \oplus \frac{S}{J_{{G_v}}}\longrightarrow \frac{S}{(x_v,y_v)+J_{G_v \setminus v}} \longrightarrow 0
     \end{align}
     and correspondingly the  long exact sequence of Tor modules:
     \begin{align}\label{ohtani-tor}
\cdots & \longrightarrow \Tor_{i}^{S}\left( \frac{S}{J_G},\K\right)_{i+j}\longrightarrow \Tor_{i}^{S}\left( \frac{S}{(x_v,y_v)+J_{G\setminus v}},\K\right)_{i+j} 
\oplus \Tor_{i}^{S}\left(\frac{S}{J_{G_v}},\K\right)_{i+j}\nonumber \\ & \longrightarrow \Tor_{i}^{S}\left(\frac{S}{(x_v,y_v)+J_{G_v\setminus v}},\K\right)_{i+j} \longrightarrow \Tor_{i-1}^{S}\left( \frac{S}{J_G},\K\right)_{i+j}\longrightarrow \cdots .
     \end{align}
\section{Graphs with $d(G)+f(G)+1=n+2-\kappa(G)$}
Let $G$ be a simple graph on the vertex set $[n]$ with diameter $d(G)$, number of simplicial vertices $f(G)$ and vertex connectivity $\kappa(G)$. The main aim of this article is to characterize graphs $G$ such that $$d(G)+f(G)+1=n+2-\kappa(G)$$ and then to compute $\depth(S/J_G)$.

In \cite[Lemma 1.1]{Hibi-Madani}, the authors proved that for a graph $G$ on $[n]$, if $d(G)\geq 3$ and $\kappa(G)\geq 2$, then $d(G)+f(G) < n+2-\kappa(G)$. We begin with an analogues result.
\begin{lemma}\label{tech-lemma}
Let $G$ be a graph on $[n]$. If $d(G) + f(G) + 1 = n+2-\kappa(G)$ and $\kappa(G) \geq 2$, then $d(G) \leq 3$.
%
	
\end{lemma}
\begin{proof}
Assume that $d(G)+f(G)+1=n+2-\kappa(G)$ and $\kappa(G) \geq 2$. For convenience, we set $d(G)=d, f(G) = f$ and $\kappa(G)=\kappa$. Then $\iv(G)=n-f=(d-2)+(\kappa+1).$ Let $u,v\in V(G)$ be such that $\dist_G(u,v)=d$. By \cite[Chapter III, Corollary 6]{Bollobas}, $u$ and $v$ can be joined by $\kappa$ vertex disjoint paths, say
$$ P_i:u,u_1^{i},\dots,u_{j_i-1}^{i},v$$
for $i = 1, \ldots, \kappa$. Note that $u_j^i$'s are internal vertices for all $i,j$. Therefore, for each $1\leq i\leq \kappa$, $j_i\geq d$, counting the internal vertices of these paths, we get $\iv(G)\geq (d-1)\kappa$. Therefore, $(d-1)\kappa\leq (d-2)+(\kappa+1)$. This implies that $(d-2)(\kappa-1)\leq 1$. Hence $d \leq 3$.
\end{proof}
The following example shows that $\kappa(G) \geq 2$ can not be removed from the previous lemma.
\begin{example}\label{kappa1}
Let $f \geq 3$ and $d \geq 2$ be integers. Let $H=K_f\cup_e K_3$. Let $v\in e$ and $G=H\cup_v P_{d}$. Then it follows that $\kappa(G)=1$, $d(G)=d$, $f(G)=f$ and $d(G)+f(G)+1=n+2-\kappa(G)$. 
\begin{minipage}{\linewidth}
\begin{minipage}{.4\linewidth}
This shows that if $\kappa(G)=1$, there the equality does not impose any condition on $d(G)$, implying that the condition $\kappa(G) \geq 2$ is necessary in \Cref{tech-lemma}.
\end{minipage}
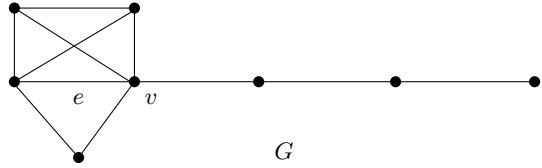
\begin{minipage}{.55\linewidth}
\captionsetup[figure]{labelformat=empty}
	\begin{figure}[H]
		\begin{tikzpicture}[scale=1.4]
		\draw (-2.06,2.56)-- (-0.92,2.56);
		\draw (-0.92,2.54)-- (-0.92,1.86);
		\draw (-0.92,1.86)-- (-2.06,1.86);
		\draw (-2.06,2.56)-- (-2.06,1.86);
		\draw (-2.06,1.86)-- (-0.92,2.54);
		\draw (-2.06,2.56)-- (-0.96,1.86);
		\draw (-0.92,1.86)-- (-1.45,1.14);
		\draw (-2.08,1.86)-- (-1.45,1.14);
		\draw (-0.92,1.86)-- (0.26,1.86);
		\draw (0.26,1.86)-- (1.56,1.86);
		\draw (1.56,1.86)-- (2.88,1.86);
		\begin{scriptsize}
		\fill  (-2.06,2.56) circle (1.5pt);
		\fill  (-0.92,2.56) circle (1.5pt);
		\draw (-0.76,1.7) node {$v$};
		\draw (-1.45,1.7) node {$e$};
		\fill  (-2.06,1.86) circle (1.5pt);
		\fill  (-0.92,1.86) circle (1.5pt);
		\fill  (-1.45,1.14) circle (1.5pt);
		\fill  (0.26,1.86) circle (1.5pt);
		\fill  (1.56,1.86) circle (1.5pt);
		\fill  (2.88,1.86) circle (1.5pt);
		\draw (0.5,1.2) node {$G$};
		\end{scriptsize}
		\end{tikzpicture}
	\end{figure}
\end{minipage}
\end{minipage}
\end{example}

Our first aim is to characterize graphs $G$ satisfying the property $d(G)+f(G)+1 = n+2-\kappa(G)$. It can be easily verified that if $n \leq 4$, then the above condition is never met. So, we assume that $n \geq 5$. 
\begin{theorem}\label{charac-thm}
Given integers $n\geq 5$, $\kappa, f, d  \in \mathbb{N}$, satisfying 
$d+f+1=n+2-\kappa,$
there exists a connected graph $G$ on $[n]$ with 
$f(G)=f,d(G)=d$ and $\kappa(G)=\kappa $
if and only if these integers satisfy one of the following conditions:
\begin{enumerate}
    \item $\kappa =1, f\geq 2$ and $d \geq 2,$
    \item $\kappa =2, f \geq 2$ and $d \in \{2, 3\},$
    \item $\kappa \geq 3, f \geq 2$ and $d=2.$
\end{enumerate}
\end{theorem}
\begin{proof}
First, observe that for a graph $G$, $d(G)=1$ if and only if $G$ is a complete graph for which the equality $d+f+1=n+2-\kappa$ does not hold. Therefore in our case $d\geq2$.

Assume that there exists a graph $G$ such that $d+f+1=n+2-\kappa$. Let $u, v \in V(G)$ such that $\dist_G(u,v) = d$ and let $P$ be a path of length $d$ joining $u$ and $v$.

Suppose $\kappa = 1$. If possible, let us assume that $f= 1$. Then $n=d+1$ and $\iv(G)=d$. Then $V(G) = V(P)$ and either $u$ or $v$ is an internal vertex. But this is not possible since $V(G)=V(P)$. Hence $f\neq 1$. Thus $\kappa =1, f\geq 2$ and $d \geq 2$. It may also be noted that the case $f = 2, d=2$ and $\kappa = 1$ is not possible since this would force $n = 4$.

Now suppose $\kappa = 2$. Then it follows from \Cref{tech-lemma} that $d \leq 3$. Suppose $f = 1$. In this case, $n=d+2$ and $\iv(G)=d+1$. So there is only one vertex in $V(G)\setminus V(P).$ Also by \cite[Chapter III, Corollary 6]{Bollobas}, there exist two vertex disjoint paths of length at least $d$ from $u$ to $v$. If $d = 3$, then $n \geq d+3$ which is not possible. Hence $d=2$. But then, $n=4$ which is a contradiction to our assumption that $n \geq 5$. Hence $f \geq 2$. Therefore, $\kappa =2, f \geq 2$ and $d \in \{2, 3\}$.

Assume that $\kappa \geq 3$. By \Cref{tech-lemma}, $d \leq 3$. We need to show  that $f \geq 2$ and $d \neq 3$.
If possible, let $d = 3$. Then $\iv(G)=\kappa+2$. By \cite[Chapter III, Corollary 6]{Bollobas}, there exist $\kappa$ vertex disjoint paths of length at least $3$ from $u$ to $v$. Each of these paths has at least two internal vertices of $G$. This implies that $2\kappa \leq \kappa+2$, and so $\kappa \leq 2$. This is a contradiction. Hence, $d = 2$. Now assume that $f = 1$. Then $d=2$, $\iv(G)=\kappa+1$ and $n=\kappa+2$. For $1 \leq i \leq \kappa$, let $P_i:u,u^i,v$ be a path from $u$ to $v$. Note that $u^1, u^2, \ldots, u^\kappa$ are distinct internal vertices. Therefore, precisely one among $u$ and $v$ is a simplicial vertex. For $u$ to be a simplicial vertex, $\{u, u^1, u^2, \ldots, u^\kappa\}$ must be a clique. This would imply that $\{v,u^1, u^2, \ldots, u^\kappa\}$ is a clique and consequently $v$ is also a simplicial vertex. This is a contradiction. Hence $f \neq 1$. Therefore $\kappa \geq 3, f\geq 2$ and $d =2$.

Conversely, suppose $(1), (2)$ or $(3)$ of the statement of the theorem is satisfied by $\kappa, f$ and $d$. Then the existence of a graph $G$ satisfying $d + f + 1  = n+2-\kappa$, follows from Examples \ref{kappa1}, \ref{kappa1_f2}, \ref{kappa2_d2}, \ref{kappa2_d3} and \ref{kappa3_d2}.
\end{proof}

\begin{example}\label{kappa1_f2}
	Let $d\geq 3$ and $\Sigma=K_{3}\cup_{e}K_{3}\cup_{v} P_{d-1}$, where $v\notin e$. 	Then  $|V(\Sigma)|=d+2$, $\kappa(\Sigma)=1$, $d(\Sigma)=d$,$f(\Sigma)=2$  and $d(\Sigma)+f(\Sigma)+1=n+2-\kappa(\Sigma)$.

	\begin{minipage}{\linewidth}
		\begin{minipage}{.55\linewidth}
		 The graph $\Sigma$, given on the right, is $\Sigma=K_3\cup_{e} K_3\cup_{v} P_3$. Note that $\kappa(\Sigma)=1$, $f(\Sigma)=2$, $d(\Sigma)=4$, and $d(\Sigma)+f(\Sigma)+1=n+2-\kappa(\Sigma)$.
		\end{minipage}
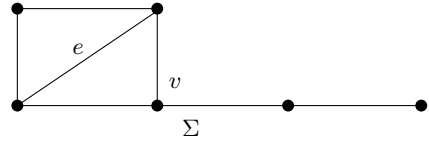
		\begin{minipage}{.4\linewidth}
			\captionsetup[figure]{labelformat=empty}
			\begin{figure}[H]
				\begin{tikzpicture}[scale=1.5]
\draw (-2.24,3.56)-- (-1,3.56);
\draw (-1,3.56)-- (-1,2.7);
\draw (-1,2.7)-- (-2.24,2.7);
\draw (-2.24,2.7)-- (-2.24,3.56);
\draw (-2.24,2.7)-- (-1,3.54);
\draw (-1,2.7)-- (0.16,2.7);
\draw (0.16,2.7)-- (1.34,2.7);
\begin{scriptsize}
\fill  (-2.24,3.56) circle (1.5pt);
\fill (-1,3.56) circle (1.5pt);
\draw (-1.7,3.2) node {$e$};
\fill (-2.24,2.7) circle (1.5pt);
\fill (-1,2.7) circle (1.5pt);
\draw (-0.84,2.9) node {$v$};
\fill (0.16,2.7) circle (1.5pt);
\fill (1.34,2.7) circle (1.5pt);
\draw (-0.7,2.5) node {$\Sigma $};
\end{scriptsize}
\end{tikzpicture}
			\end{figure}
		\end{minipage}
	\end{minipage}
\end{example}

\begin{example}\label{kappa2_d2}
	Let $f\geq 2$ and $\Gamma=K_{3}\cup_{e}K_{f+1}\cup_{e'} K_{3}, \text{ where }  \mid e\cap e' \mid ~= 1.$ 		Then  $|V(\Gamma)|=f+3$, $\kappa(\Gamma)=2$, $d(\Gamma)=2$,$f(\Gamma)=f$  and $d(\Gamma)+f(\Gamma)+1=n+2-\kappa(\Gamma)$.

	\begin{minipage}{\linewidth}
		\begin{minipage}{.55\linewidth}
 The graph $\Gamma$, given on the right, is $\Gamma=K_3\cup_{e} K_4\cup_{e'} K_3$. Note that $\kappa(\Gamma)=2$, $f(\Gamma)=3$, $d(\Gamma)=2$, and $d(\Gamma)+f(\Gamma)+1=n+2-\kappa(\Gamma)$.
		\end{minipage}
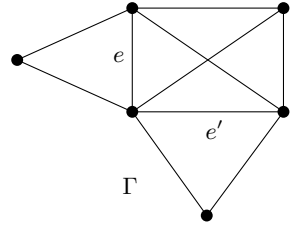
		\begin{minipage}{.4\linewidth}
			\captionsetup[figure]{labelformat=empty}
			\begin{figure}[H]
				\begin{tikzpicture}[scale=1.5]
	\draw (-1.5,3.72)-- (-0.48,4.18);
	\draw (-0.48,4.18)-- (-0.48,3.26);
	\draw (-1.5,3.72)-- (-0.48,3.26);
	\draw (-0.48,4.18)-- (0.86,4.18);
	\draw (0.86,4.18)-- (0.86,3.26);
	\draw (0.86,3.26)-- (-0.48,3.26);
	\draw (-0.48,3.26)-- (0.86,4.18);
	\draw (-0.48,4.18)-- (0.86,3.26);
	\draw (0.86,3.26)-- (0.18,2.34);
	\draw (-0.48,3.26)-- (0.18,2.34);
	\begin{scriptsize}
	\fill  (-1.5,3.72) circle (1.5pt);
	\fill  (-0.48,4.18) circle (1.5pt);
	\fill  (-0.48,3.26) circle (1.5pt);
	\draw (-0.6,3.75) node {$e$};
	\fill  (0.86,4.18) circle (1.5pt);
	\fill  (0.86,3.26) circle (1.5pt);
	\draw (0.25,3.1) node {$e'$};
	\fill  (0.18,2.34) circle (1.5pt);
	\draw (-0.5,2.6) node {$\Gamma$};
	\end{scriptsize}
				\end{tikzpicture}
			\end{figure}
		\end{minipage}
	\end{minipage}
\end{example}
\begin{example}\label{kappa2_d3}
	Let $f\geq 2$ and $\Omega=K_{3}\cup_{e}K_{f+2}\cup_{e'} K_{3}, \text{ where } e\cap e'=\emptyset .$
	Then $|V(G)|=f+4$, $\kappa(G)=2$, $d(\Omega)=3$, $f(\Omega)=f$, and $d(\Omega)+f(\Omega)+1=n+2-\kappa(\Omega)$.
	
	\begin{minipage}{\linewidth}
		\begin{minipage}{.4\linewidth}
			The graph $\Omega$, given on the right, is $\Omega=K_3\cup_{e} K_4\cup_{e'} K_3$. Note that $\kappa(\Omega)=2$, $f(\Omega)=2$, $d(\Omega)=3$, and $d(\Omega)+f(\Omega)+1=n+2-\kappa(\Omega)$.
		\end{minipage}
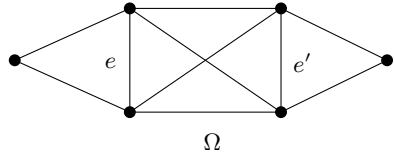
		\begin{minipage}{.6\linewidth}
			\captionsetup[figure]{labelformat=empty}
			\begin{figure}[H]
				\begin{tikzpicture}[scale=1.5]
				\draw (-1.5,3.72)-- (-0.48,4.18);
				\draw (-0.48,4.18)-- (-0.48,3.26);
				\draw (-1.5,3.72)-- (-0.48,3.26);
				\draw (-0.48,4.18)-- (0.86,4.18);
				\draw (0.86,4.18)-- (0.86,3.26);
				\draw (0.86,3.26)-- (-0.48,3.26);
				\draw (-0.48,3.26)-- (0.86,4.18);
				\draw (-0.48,4.18)-- (0.86,3.26);
				\draw (0.86,4.18)-- (1.8,3.72);
				\draw (0.86,3.26)-- (1.8,3.72);
				\begin{scriptsize}
				\fill  (-1.5,3.72) circle (1.5pt);
				\fill  (-0.48,4.18) circle (1.5pt);
				\fill  (-0.48,3.26) circle (1.5pt);
				\draw (-0.65,3.7) node {$e$};
				\fill  (0.86,4.18) circle (1.5pt);
				\fill  (0.86,3.26) circle (1.5pt);
				\draw (1.05,3.7) node {$e'$};
				\draw (0.25,3) node {$\Omega$};
				\fill  (1.8,3.72) circle (1.5pt);
				\end{scriptsize}
				\end{tikzpicture}
			\end{figure}
		\end{minipage}
	\end{minipage}
\end{example}

\begin{example}\label{kappa3_d2}
    Let $f\geq 2,\kappa \geq 3$ and 
    $\Delta =K_{\kappa +1} \cup_{K_\kappa }K_{f+\kappa -1}\cup_{K'_\kappa }K_{\kappa +1}, \text{ where } |V(K_\kappa )\cap V(K'_\kappa)|=\kappa-1.$
    Note that $V(\Delta)=f+\kappa +1$, $d(\Delta)=2$, $f(\Delta)=f$, $\kappa(\Delta)=\kappa$ and $d(\Delta)+f(\Delta)+1=n+2-\kappa(\Delta).$
    
    \begin{minipage}{\linewidth}
		\begin{minipage}{.4\linewidth}
			The graph, given on the right, is $\Delta=K_4\cup_{K_3} K_4\cup_{K'_3}  K_4$, where $V(K_3)=\{i,j,l\}$ and $V(K'_3)=\{j,k,l\}$. Note that $\kappa(\Delta)=3$, $f(\Delta)=2$, $d(\Delta)=2$, and $d(\Delta)+f(\Delta)+1=n+2-\kappa(\Delta)$.
		\end{minipage}
		\begin{minipage}{.6\linewidth}
			\captionsetup[figure]{labelformat=empty}
			\begin{figure}[H]
				\begin{tikzpicture}[scale=1.5]
				\draw (-1.5,3.72)-- (-0.48,4.18);
				\draw (-0.48,4.18)-- (-0.48,3.26);
				\draw (-1.5,3.72)-- (-0.48,3.26);
				\draw (-0.48,4.18)-- (0.86,4.18);
				\draw (0.86,4.18)-- (0.86,3.26);
				\draw (0.86,3.26)-- (-0.48,3.26);
				\draw (-0.48,3.26)-- (0.86,4.18);
				\draw (-0.48,4.18)-- (0.86,3.26);
				\draw (0.86,4.18)-- (1.8,3.72);
				\draw (0.86,3.26)-- (1.8,3.72);
				\draw (-1.5, 3.72)-- (0.86,4.18);
				\draw (1.8, 3.72)-- (-0.48,3.26);
			
				\begin{scriptsize}
				\fill  (-1.5,3.72) circle (1.5pt);
				\fill  (-0.48,4.18) circle (1.5pt);
				\draw (-0.32,4.42) node {$i$};
				\fill  (-0.48,3.26) circle (1.5pt);
				\draw (-0.49,3) node {$j$};
				\fill  (0.86,4.18) circle (1.5pt);
				\draw (1.06,4.42) node {$l$};
				\fill  (0.86,3.26) circle (1.5pt);
				\draw (1.05,3) node {$k$};
				\draw (0.2,3) node {$\Delta$};
				\fill  (1.8,3.72) circle (1.5pt);
				\end{scriptsize}
				\end{tikzpicture}
			\end{figure}
		\end{minipage}
	\end{minipage}
\end{example}

Using \Cref{charac-thm}, we now give some necessary conditions for a graph to satisfy the property $d(G)+f(G)+1 = n+2-\kappa(G)$.

\begin{theorem}\label{induced-cycle-charac}
Let $G$ be a graph on $[n]$ such that $d(G)+f(G)+1=n+2-\kappa(G).$ 
\begin{enumerate}
    \item If $\kappa(G)=1$, then either $G$ is chordal or $G$ has precisely one induced $C_4$ and has no induced $C_l$ for $l \geq 5$. 
    \item If $\kappa(G)\geq 2$ and $d(G)=2$, then $G$ is a chordal graph.
    \item If $\kappa(G)=2$ and $d(G)=3$, then either $G$ is  chordal or $G$ has precisely one induced $C_4$ and has no induced $C_l$ for $l \geq 5$.
\end{enumerate}
\end{theorem}
\begin{proof}
First we treat the case when $\kappa(G)=1$. If $d(G) + f(G) + 1 = n+2-\kappa(G)$, then $\kappa(G) =  1$ if and only $d(G) + f(G) = n$.  Hence we study graphs $G$ with $d(G)+f(G)=n$. Set $d=d(G)$ and $\kappa = \kappa(G)$. Then, $\iv(G)=n-f(G)=d.$ Let $u,v\in V(G)$ be such that $\dist_G(u,v)=d$. Let $P:u=u_0,u_1,\dots,u_{d-1},u_d=v$ be an induced path from $u$ to $v$. Since $u_1, \ldots, u_{d-1}$ are internal vertices, there is precisely one more internal vertex in $[n]\setminus \{u_1,\dots,u_{d-1}\}.$ This can happen in two ways:

\textbf{Case-I:} Suppose either $u$ or $v$ is an internal vertex. Without loss of generality we assume that $u$ is an internal vertex.
Suppose $G$ contains an induced cycle of length $l\geq 4$, say $C_l$. Let $V(C_l)=\{v_1,\dots,v_l\}$. Since $v_i$'s are internal vertices, $\{v_1,\dots,v_l\}\subseteq \{u,u_1,\dots,u_{d-1}\}$. This implies that either there exists $3\leq i \leq d-1$ such that $\{u,u_i\}\in E(G)$ or there exist integers $j, k$, $k > j+2$ such that $\{u_j, u_k\} \in E(G)$. In either case,  this would force that $\dist_G(u,v) < d$ which is a contradiction. Therefore, $G$ can not contain an induced cycle of  length $\geq 4$. So, $G$ is a chordal graph.

\textbf{Case-II:} Suppose neither $u$ nor $v$ is an internal vertex. In this case we prove that either $G$ is chordal or $G$ can have at most one induced cycle of length 4 and can have no induced cycle of length bigger than 4. Assume that $G$ is not chordal. Suppose $G$ contains an induced cycle of length of $l\geq 5$, say $C_l$. Let $V(C_l)=\{v_1,\dots,v_l\}$. Since $v_i$'s are internal vertices, $u_1, \ldots, u_{d-1}$ is a path and outside $V(P)$ there is exactly one internal vertex, $|\{v_1,\dots,v_l\}\cap \{u_1,\dots,u_{d-1}\}|=l-1$. Without loss of generality, we may assume that $v_l\notin \{u_1,\dots,u_{d-1}\}$ and $u_i = v_i$ for $i=1, \ldots, l-1$. 

\noindent
\begin{minipage}{\linewidth}
		\begin{minipage}{.4\linewidth}
		Then $P':u,v_1,v_l,v_{l-1},u_l,\ldots,u_{d-1},v$ is a path of length $d-(l-2)+2= d-l+4 \leq d-1$ from $u$ to $v$ which contradicts the assumption that $\dist_G(u,v) = d$.
		\end{minipage}
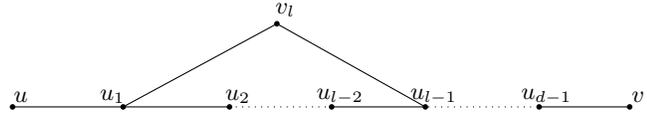
		\begin{minipage}{.6\linewidth}
    \captionsetup[figure]{labelformat=empty}
			
			\begin{figure}[H]
\begin{tikzpicture}[scale=0.7]
\draw (-2.6,1.5)-- (-0.5,1.5);
\draw (-0.5,1.5)-- (1.52,1.5);
\draw[dotted] (1.52,1.5)-- (3.46,1.5);
\draw (5.24,1.5)-- (3.46,1.5);
\draw (7.4,1.5)-- (9.12,1.5);
\draw (-0.5,1.5)-- (2.42,3.08);
\draw (2.42,3.08)-- (5.24,1.5);
\draw[dotted] (5.3, 1.5) -- (7.3,1.5);
\begin{scriptsize}
\fill  (-2.6,1.5) circle (1.5pt);
\draw (-2.44,1.7) node {$u$};
\fill  (-0.5,1.5) circle (1.5pt);
\draw (-0.7,1.7) node {$u_1$};
\fill  (1.52,1.5) circle (1.5pt);
\draw (1.68,1.7) node {$u_2$};
\fill  (3.46,1.5) circle (1.5pt);
\draw (3.62,1.7) node {$u_{l-2}$};
\fill  (5.24,1.5) circle (1.5pt);
\draw (5.4,1.7) node {$u_{l-1}$};
\fill  (7.4,1.5) circle (1.5pt);
\draw (7.54,1.7) node {$u_{d-1}$};
\fill  (9.12,1.5) circle (1.5pt);
\draw (9.28,1.7) node {$v$};
\fill  (2.42,3.08) circle (1.5pt);
\draw (2.58,3.34) node {$v_l$};
\end{scriptsize}
				\end{tikzpicture}
			\end{figure}
			\end{minipage}
	\end{minipage}
Therefore, $G$ does not have an induced cycle of length $5$ or more.
\vskip 4mm
Now if possible assume that $G$ has at least two induced cycles of length 4. Since $G$ has one internal vertex belonging to $[n]\setminus \{u_1,\dots,u_{d-1}\}$, all the vertices of one of these two cycles is a subset of $\{u_1,\dots,u_{d-1}\}$. This would again force that $\dist_G(u,v) \leq d-2$, which is a contradiction. Hence if $G$ is not chordal, then $G$ contains at most one induced $C_4$. This completes the proof of $(1)$.

Now we treat the case $\kappa(G)\geq 2$. Then it follows from \Cref{charac-thm} that $2\leq d\leq 3$, and $\iv(G)=n-f(G)=d-1+\kappa.$ Suppose $d = 2$. Then $\iv(G)=\kappa+1 $. Assume that $G$ has an induced cycle $C_l$ $(l\geq 4)$ with $V(C_l)=\{v_1,\dots,v_l\}$. By \cite[Chapter III, Corollary 6]{Bollobas}, there exist $\kappa$ vertex-disjoint paths of length $\geq 2$ from $v_1$ to $v_3$. Each such path contains at least one internal vertex. Since $v_1$ and $v_3$ are vertices of an induced cycle of length $\geq 4$, these are internal vertices as well. Hence $\iv(G) \geq \kappa + 2$. This is a contradiction. Hence, $G$ is chordal. This proves $(2)$.

Assume that $d=3$. Then it follows from Theorem \ref{charac-thm} that $\kappa=2$, and so $\iv(G)=4$. Since all the vertices of cycles are internal vertices, $G$ has no induced cycle of length $\geq 5$ and $G$ can have at most one induced cycle of length 4. This completes the proof of the theorem.
\end{proof}

\section{Depth}
In this section, we compute the depth of $S/J_G$ for graphs satisfying $d(G)+f(G)+1 = n+2-\kappa(G)$. In this computation, the graphs $G_v, G_v \setminus v$ and $G\setminus v$, for some internal vertex $v$, play important roles via the Ohtani Lemma, (\Cref{ohtani-lemma}) and \eqref{ohtani-ses}. We begin this section with a couple of technical lemmas which are useful in proving the main results.
\begin{lemma}\label{g-v}
Let $G$ be a graph and $v$ be an internal vertex. Then,
\begin{enumerate}
    \item if $G \setminus v$ is connected, then $d(G) \leq d(G\setminus v)$;
    \item $f(G) \leq f(G\setminus v)$.
\end{enumerate}
\end{lemma}
\begin{proof}
For any two vertices $z, z'$ in $G\setminus v$, a shortest induced path in $G\setminus v$ joining $z$ and $z'$ will be a path in $G$, but need not necessarily be an induced path. This implies that $\dist_{G}(z,z') \leq \dist_{G\setminus v}(z,z')$. Hence $d(G) \leq d(G\setminus v)$. This proves (1).

(2) follows directly from \cite[Lemma 3.2]{Arv-Jaco}.
\end{proof}
We now prove a result connecting the diameter of a graph $G$ and the diameter of $G_v$ for some internal vertex $v$.

\begin{lemma}\label{diameter-lemma}
Let $G$ be any graph and $v$ be any internal vertex in $G$. Then for any two vertices $z,z'\in V(G)$, $$\dist_G(z,z')-1\leq \dist_{G_v}(z,z') \leq \dist_{G}(z,z').$$
In particular, $d(G)-1 \leq d(G_v) \leq d(G).$
\end{lemma}
\begin{proof}
Note that an induced path in $G$ need not necessarily be an induced path in $G_v$. Thus, for any $z,z'\in V(G)$, $\dist_{G_v}(z,z')\leq \dist_{G}(z,z')$. 
Let $P:z=z_1,\dots,z_{d'},z_{d'+1}=z'$ be a shortest induced path from $z$ to $z'$ in $G_v$.  
If $E(P)\subseteq E(G)$, then $P$ remains an induced path in $G$. Hence, $\dist_G(z,z')\leq d'$. Observe that all the edges in $E(G_v)\setminus E(G)$ belong to a single clique in $G_v$. If $E(P)\nsubseteq E(G)$, then there exists only one edge $e\in E(P)\setminus E(G)$. Let $e=\{z_i,z_{i+1}\}$ for some $1\leq i\leq d'$ and $z_i,z_{i+1}\in N_G(v)$. Hence, in $G$, we have an induced path $z=z_1, \ldots, z_i, v, z_{i+1}, \ldots, z_{d'}=z'$ in $G$. Therefore $\dist_G(z,z') \leq d'+1$.
\end{proof}
It is easy to observe that the structure of $G_v$ and $G_v\setminus v$ are similar and the inequalities in the above lemma remain the same if we replace $G_v$ by $G_v \setminus v$.

Let $G$ be a graph on $[n]$ such that $d(G)+f(G)=n+2-\kappa(G)$. Hibi and Saeedi Madani classified these graphs into two collection, $\mathcal{G}_d$ and $\mathcal{F}_q$ (see \cite[Section 2]{Hibi-Madani}). It follows from \cite[Corollary 2.5]{Hibi-Madani} and Auslander-Buchsbaum formula that if $G$ is either in $\mathcal{G}_d$ or in $\mathcal{F}_q$, then $\pd(S/J_G)=n-2+\kappa(G)$. 

One of the main tools in the study of depths of binomial edge ideal is the so-called Ohtani short exact sequence, \eqref{ohtani-ses}. While studying graphs $G$ with $d(G)+f(G)+1=n+2-\kappa(G)$, we see that some of the graphs $G_v$ or $G\setminus v$ happens to be a graph in the Hibi-Madani class. Therefore, in order to study the depth, it becomes necessary to understand certain extremal Betti numbers of graphs in the Hibi-Madani class.
For a finitely generated graded $S$-module $M$, if $\beta_{i,i+j}(M)\neq 0$ and for all pairs $(k,l)\neq (i,j)$ with $k\geq i$ and $l\geq j$, $\beta_{k,k+l}(M)=0$, then $\beta_{i,i+j}(M)$ is called an \textit{extremal Betti number} of $M$. If $p =\pd(M)$, then there exists a unique number $i$ such that $\beta_{p,p+i}(M)$ is an extremal Betti number.

\begin{proposition}\label{extremal-hibi-madani}
Let $G$ be a graph $[n]$ such that $d(G)+f(G)=n+2-\kappa(G)$. Then $\beta_{n-2+\kappa(G),n-2+\kappa(G)+d(G)}(S/J_G)$ is an extremal Betti number of $S/J_G$. 
\end{proposition}
\begin{proof}
First we consider the case when $\kappa(G)=1$. Then $G\in \mathcal{G}_{d(G)}$. We proceed by induction on $d(G)$. If $d(G) = 2$, then $G$ is a block graph and the assertion follows from \cite[Theorem 6]{her2}, noting that $\iv(G)=d(G)-1$. Assume that $d(G)\geq 3$. Let $w$ be an internal vertex. Then it follows from \Cref{diameter-lemma} that $d(G_w) \geq d(G)-1$. From \cite[Lemma 3.2]{Arv-Jaco}, we get $f(G_w) \geq f(G)+1$. Therefore, $n+1=d(G)+f(G) \leq d(G_w)+f(G_w)\leq n+1$. The same way we can see that $d(G_w\setminus w) + f(G_w\setminus w) = n$. Hence $G_w,G_w\setminus w\in \mathcal{G}_{d-1}$. Also, $G\setminus w$ is a disconnected graph where each component of $G\setminus w$ is either a complete graph or belongs to $\mathcal{G}_i$ for some $i\leq d-2$. Therefore, $\depth(S/J_{G_w})=n+1,\depth(S/((x_w,y_w)+J_{G_w\setminus w}))=n$ and $\depth(S/((x_w,y_w)+J_{G\setminus w}))\geq n+1$. Hence by Auslander-Buchsbaum formula, we have $\pd(S/J_{G_w})=n-1,\pd(S/((x_w,y_w)+J_{G_w\setminus w}))=n$ and $\pd(S/((x_w,y_w)+J_{G\setminus w}))\leq n-1$. Also, note that $d(G_w)=d(G_w\setminus w)=d(G)-1$.  Therefore, by induction, $\beta_{n-1,n-1+d(G_w)}(S/J_{G_w})$, $\beta_{n,n+d(G_w\setminus w)}(S/((x_w,y_w)+J_{G_w\setminus w}))$ are extremal Betti numbers. 
If $\pd(S/((x_w,y_w)+J_{G\setminus w})) < n-1$, then $\beta_{n-1,j}(S/((x_w,y_w)+J_{G\setminus w})) = 0$ for all $j$. If $\pd(S/((x_w,y_w)+J_{G\setminus w}))= n-1$, then $G\setminus w$ has precisely two components, say $G_1$ and $G_2$. Set $|V(G_1)|=n_1$ and $|V(G_2)|=n_2$. Since $d(G_1),d(G_2) < d$, by applying induction on $G_1$ and $G_2$, we get that $\beta_{n_i-1,n_i-1+d(G_i)}(S_{G_i}/J_{G_i})$ is an extremal Betti number for $i=1,2$. Observe that $d(G_1)+d(G_2)\leq d(G)$. Hence by \eqref{Bettiproduct},  $\beta_{n-1,n-1+j}(S/((x_w,y_w)+J_{G\setminus w}))=0$ for $j\geq d(G)+1$. Therefore, it follows from \eqref{ohtani-tor} for $i=n-1$ and $j=d(G)$ that
that $$\beta_{n-1,n-1+d(G)}(S/J_G)\neq 0 \text{ and } \beta_{n-1,n-1+j}(S/J_G)=0 \text{ for } j\geq d(G)+1. $$
Since $\pd(S/J_G)=n-1$, $\beta_{n-1,n-1+d(G)}(S/J_G)$ is an extremal Betti number of $S/J_G$.

Now we consider $\kappa(G)\geq 2$. Then $G\in \mathcal{F}_{\kappa(G)}$ with $d(G)=2$. Take $w$ to be an internal vertex in $G$. Then it can be noted that $G_w, G_w\setminus w$ are complete graphs on $n$ and $n-1$ vertices respectively. So, $\beta_{n-1,n}(S/J_{G_w})$ and $\beta_{n,n+1}(S/((x_w,y_w)+J_{G_w\setminus w}))$ are extremal Betti numbers. We proceed by induction on $\kappa(G)$. Assume that $\kappa(G)=2$. Then $G\setminus w$ is a block graph on $n-1$ vertices, and hence by \cite[Theorem 6]{her2}, $\beta_{n,n+2}(S/((x_w,y_w)+J_{G\setminus w}))$ is an extremal Betti number as $\iv(G\setminus w)=1$. Therefore it follows from the long exact sequence \eqref{ohtani-tor} that $\beta_{n,n+2}(S/J_G)$ is an extremal Betti number of $S/J_G$. Suppose now $\kappa(G)\geq 3$. In this case, note that $G\setminus w\in \mathcal{F}_{\kappa(G)-1}$. Therefore, by induction and \eqref{Bettiproduct}, $\beta_{n-1+\kappa(G\setminus w),n-1+\kappa(G\setminus w)+d(G\setminus w)}(S/((x_w,y_w)+J_{G\setminus w}))$ is an extremal Betti number. Since $\kappa(G\setminus w)=\kappa(G)-1$ and $d(G\setminus w)=d(G)$, the assertion follows from the long exact sequence \eqref{ohtani-tor}.
\end{proof}

For the rest of the section, $G$ denotes a connected non-complete graph on $[n]$ with $d(G)+f(G)+1= n+2-\kappa(G)$. We now describe the structure of these graphs in more details and then compute the depth of $S/J_G$.
\vskip 2mm
Assume that $\kappa(G)=1$. Then $d(G)+f(G)=n$. Let $u, v \in V(G)$ such that $\dist_G(u,v) = d(G) =d$. Let $P: u, u_1, \ldots, u_{d-1}, v$ be an induced path from $u$ to $v$. Since $\iv(G) = n-f(G) = d(G)$ and $u_1, \ldots, u_{d-1}$ are internal vertices, there  is precisely one internal vertex among $[n]\setminus \{u_1, \ldots, u_{d-1}\}$. This can occur in the following ways:
\begin{enumerate}
    \item either $u$ or $v$ is an internal vertex;
    \item there is an internal vertex $v' \in [n] \setminus V(P)$.
\end{enumerate}
If either $u$ or $v$ is an internal vertex, then it follows from the proof of \Cref{induced-cycle-charac} that $G$ is a chordal graph having all its internal vertices contained in $V(P)$. If both $u$ and $v$ are simplicial vertices, then there is a unique internal vertex in $[n]\setminus V(P)$, say $v'$. Since $v'$ is the unique internal vertex in $[n] \setminus V(P)$, $\dist(v', P) = 1$. Let $j = \min\{\ell : \{u_\ell, v'\} \in E(G)\}$. It may also be noted that $\max\{k : u_k \in N_P(v')\} \leq j+2$, since otherwise it will contradict the assumption that $\dist_G(u,v) = d$. Hence $N_P(v') \subseteq \{u_j, u_{j+1},u_{j+2}\}$. 
Consider the following graph $H$:

			
\begin{figure}[H]
\centering
\begin{tikzpicture}[scale=1]
\draw[dotted] (-6, 0) -- (-5,0);
\draw[dotted] (-3, 0) -- (-2,0);
\draw (-7,0)-- (-6,0);
\draw  (-5,0)-- (-4,0);
\draw  (-4,0)-- (-3,0);
\draw  (-4,1)-- (-5,0);
\draw  (-4,1)-- (-4,0);
\draw  (-4,1)-- (-3,0);
\draw  (-2,0)-- (-1,0);
\begin{scriptsize}
\fill  (-7,0) circle (1.5pt);
\draw (-6.84,-0.3) node {$u$};
\fill  (-6,0) circle (1.5pt);
\draw (-5.84,-0.3) node {$u_1$};
\fill  (-5,0) circle (1.5pt);
\draw (-4.84,-0.3) node {$u_j$};
\fill  (-4,0) circle (1.5pt);
\draw (-3.84,-0.3) node {$u_{j+1}$};
\fill  (-3,0) circle (1.5pt);
\draw (-2.84,-0.3) node {$u_{j+2}$};
\fill (-2,0) circle (1.5pt);
\draw (-1.84,-0.3) node {$u_{d-1}$};
\fill  (-1,0) circle (1.5pt);
\draw (-0.84,-0.3) node {$v$};
\fill  (-4,1) circle (1.5pt);
\draw (-4,1.3) node {$v'$};
\end{scriptsize}
				\end{tikzpicture}
				\caption{$H$}\label{figure-H}
			\end{figure}

Note that if either $u$ or $v$ is an internal vertex of $G$, then $H$ cannot be an induced subgraph of $G$.

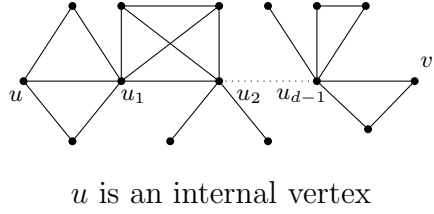
\captionsetup[figure]{labelformat=empty}
			\begin{figure}[H]
				\begin{tikzpicture}[scale=1]
				\draw (-0.65,1)-- (-1.3,0);
\draw (0,0)-- (-0.65,1);
\draw (-1.3,0)-- (-0.65,-0.8);
\draw (-0.65,-0.8)-- (0,0);
\draw (-1.3,0)-- (0,0);
\draw (0,0)-- (1.3,0);
\draw (2.6,0)-- (3.9,0);
\draw (1.3,0)-- (0.65,-0.8);
\draw (1.3,0)-- (1.95,-0.8);
\draw (2.6,0)-- (3.28,-0.64);
\draw (3.9,0)-- (3.28,-0.64);
\draw (2.6,1)-- (2.6,0);
\draw (2.6,1)-- (3.25,1);
\draw (3.25,1)-- (2.6,0);
\draw (2.6,0)-- (1.95,1);
\draw (0,1)-- (1.3,1);
\draw (1.3,1)-- (1.3,0);
\draw (0,0)-- (0,1);
\draw (0,0)-- (1.3,1);
\draw (0,1)-- (1.3,0);
\draw[dotted] (1.3, 0) -- (2.6,0);
\begin{scriptsize}
\fill  (0,0) circle (1.5pt);
\draw (0.16,-0.2) node {$u_1$};
\fill (-1.3,0) circle (1.5pt);
\draw (-1.4,-0.2) node {$u$};
\fill (1.3,0) circle (1.5pt);
\draw (1.7,-0.2) node {$u_2$};
\fill  (2.6,0) circle (1.5pt);
\draw(2.4,-0.2) node {$u_{d-1}$};
\fill  (3.9,0) circle (1.5pt);
\draw (4.06,0.26) node {$v$};
\fill  (-0.65,1) circle (1.5pt);
\fill  (-0.65,-0.8) circle (1.5pt);
\fill  (0,1) circle (1.5pt);
\fill  (1.3,1) circle (1.5pt);
\fill  (0.65,-0.8) circle (1.5pt);
\fill  (1.95,-0.8) circle (1.5pt);
\fill (1.95,1) circle (1.5pt);
\fill  (2.6,1) circle (1.5pt);
\fill  (3.25,1) circle (1.5pt);
\fill  (3.28,-0.64) circle (1.5pt);
\end{scriptsize}
				\end{tikzpicture}
				\caption{$u$ is an internal vertex}
			\end{figure}

If both $u$ and $v$ are simplicial vertices, then $H \setminus \left\{ \{v', u_{j+1}\}, \{v', u_{j+2}\}\right\}$ is necessarily an induced subgraph of $G$. It follows from \cite[Theorem 2.7]{Rinaldo-Rauf} and \cite[Proposition 3.4 ]{Rajib} that $\depth (S_H/J_{H}) = |V(H)|.$
Let \[
\mathcal{D} := \big\{ G ~: G \text{ is chordal}, \kappa(G) = 1 \text{ and } d(G)+f(G) = n \big\}
\]
\[
\mathcal{D}_1 := \left\{ G \in \mathcal{D}:
\begin{array}{l}
            H \text{ is an induced subgraph of } G \text{ and } 
      \{v'\}, \{u_{j+1}\} \notin \mathcal{C}(G)
\end{array}
\right\}.\]

For the benefit of the readers, we illustrate the graphs using some figures.

\noindent

\begin{minipage}{\linewidth}
		\begin{minipage}{.5\linewidth}
		\captionsetup[figure]{labelformat=empty}
		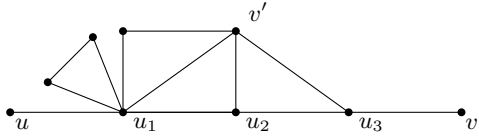
\begin{figure}[H]
				\begin{tikzpicture}[scale=1]
				\draw (-3.5,3.16)-- (-2,3.16);
\draw (-2,3.16)-- (-0.5,3.16);
\draw (-0.5,3.16)-- (1,3.16);
\draw (1,3.16)-- (2.5,3.16);
\draw (-2,3.16)-- (-2,4.24);
\draw (-2,3.16)-- (-0.5,3.16);
\draw (-2,4.24)-- (-0.5,4.24);
\draw (-0.5,4.24)-- (-0.5,3.16);
\draw (-2,3.16)-- (-0.5,4.24);
\draw (-0.5,4.24)-- (1,3.16);
\draw (-2.4,4.16)--(-2,3.16);
\draw (-3,3.56)--(-2,3.16);
\draw (-2.4,4.16)--(-3,3.56);
\begin{scriptsize}
\fill  (-2.4,4.16) circle (1.5pt);
\fill  (-3,3.56) circle (1.5pt);
\fill  (-3.5,3.16) circle (1.5pt);
\draw (-3.34,3) node {$u$};
\fill  (-2,3.16) circle (1.5pt);
\draw (-1.7,3) node {$u_1$};
\fill  (-0.5,3.16) circle (1.5pt);
\draw (-0.2,3) node {$u_2$};
\fill  (1,3.16) circle (1.5pt);
\draw (1.3,3) node {$u_3$};
\fill  (2.5,3.16) circle (1.5pt);
\draw (2.64,3) node {$v$};
\fill  (-2,4.24) circle (1.5pt);
\fill  (-0.5,4.24) circle (1.5pt);
\draw (-0.2,4.5) node {$v'$};
\end{scriptsize}
				\end{tikzpicture}
				\caption{$G_1\in \mathcal{D}_1$}
			\end{figure}
		\end{minipage}
		\begin{minipage}{.5\linewidth}
    \captionsetup[figure]{labelformat=empty}
			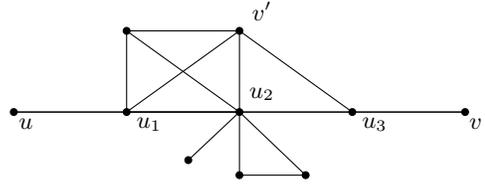
\begin{figure}[H]
				\begin{tikzpicture}[scale=1]
				
				\draw (-3.5,3.16)-- (-2,3.16);
\draw (-2,3.16)-- (-0.5,3.16);
\draw (-0.5,3.16)-- (1,3.16);
\draw (1,3.16)-- (2.5,3.16);
\draw (-2,3.16)-- (-2,4.24);
\draw (-3.5,3.16)-- (2.5,3.16);
\draw (-2,3.16)-- (-0.5,3.16);
\draw (-2,4.24)-- (-0.5,4.24);
\draw (-0.5,4.24)-- (-0.5,3.16);
\draw (-2,4.24)-- (-0.5,3.16);
\draw (-2,3.16)-- (-0.5,4.24);
\draw (-0.5,4.24)-- (1,3.16);
\draw (-0.5,3.16)-- (-0.5,2.32);
\draw (-0.5,2.32)-- (0.38,2.32);
\draw (-0.5,3.16)-- (0.38,2.32);
\draw (-0.5,3.16)-- (-1.18,2.52);
\begin{scriptsize}
\fill  (-3.5,3.16) circle (1.5pt);
\draw (-3.34,3) node {$u$};
\fill  (-2,3.16) circle (1.5pt);
\draw (-1.7,3) node {$u_1$};
\fill  (-0.5,3.16) circle (1.5pt);
\draw (-0.2,3.4) node {$u_2$};
\fill  (1,3.16) circle (1.5pt);
\draw (1.3,3) node {$u_3$};
\fill  (2.5,3.16) circle (1.5pt);
\draw (2.64,3) node {$v$};
\fill  (-2,4.24) circle (1.5pt);
\fill  (-0.5,4.24) circle (1.5pt);
\fill  (-0.5,2.32) circle (1.5pt);
\fill  (-1.18,2.52) circle (1.5pt);
\draw (-0.2,4.5) node {$v'$};
\fill  (0.38,2.32) circle (1.5pt);
\end{scriptsize}
				\end{tikzpicture}
				\caption{$G_2\in \mathcal{D}\setminus \mathcal{D}_1$ with $H$ as an induced subgraph and $\{u_2\}\in \mathcal{C}(G)$}
			\end{figure}
			
			\end{minipage}
	\end{minipage}

\begin{minipage}{\linewidth}
		\begin{minipage}{.5\linewidth}
		\captionsetup[figure]{labelformat=empty}
		\begin{figure}[H]
				\begin{tikzpicture}[scale=1]
				\draw (-3.5,3.16)-- (-2,3.16);
\draw (-2,3.16)-- (-0.5,3.16);
\draw (-0.5,3.16)-- (1,3.16);
\draw (1,3.16)-- (2.5,3.16);
\draw (-2,3.16)-- (-2,2.14);
\draw (-2,2.14)-- (-0.5,2.14);
\draw (-0.5,3.16)-- (-0.5,2.14);
\draw (-3.5,3.16)-- (-2.72,4.24);
\draw (-2.72,4.24)-- (-2,3.16);
\draw (-2,3.16)-- (-2,4.24);
\draw (-2,3.16)-- (-1.2,4.24);
\draw (-1.2,4.24)-- (-0.5,3.16);
\draw (-0.5,3.16)-- (0.3,4.24);
\draw (0.3,4.24)-- (1.18,4.24);
\draw (0.3,4.24)-- (1.1,5.08);
\draw (1.1,5.08)-- (0.32,5.08);
\draw (0.32,5.08)-- (0.3,4.24);
\draw (0.3,4.24)-- (-0.44,5.08);
\draw (-2,3.16)-- (-0.5,2.14);
\draw (-2,2.14)-- (-0.5,3.16);
\draw (-3.5,3.16)-- (2.5,3.16);
\draw (-2,3.16)-- (-0.5,3.16);
\begin{scriptsize}
\fill (-3.5,3.16) circle (1.5pt);
\draw (-3.34,3) node {$u$};
\fill  (-2,3.16) circle (1.5pt);
\draw (-2.2,3) node {$u_1$};
\fill  (-0.5,3.16) circle (1.5pt);
\draw (-0.2,3) node {$u_2$};
\fill  (1,3.16) circle (1.5pt);
\draw (1.3,3) node {$u_3$};
\fill  (2.5,3.16) circle (1.5pt);
\draw (2.64,3) node {$v$};
\fill  (-2.72,4.24) circle (1.5pt);
\fill  (-2,2.14) circle (1.5pt);
\fill  (-0.5,2.14) circle (1.5pt);
\fill  (-1.2,4.24) circle (1.5pt);
\fill  (-2,4.24) circle (1.5pt);
\fill  (0.3,4.24) circle (1.5pt);
\draw (0.3,4) node {$v'$};
\fill  (-0.44,5.08) circle (1.5pt);
\fill  (0.32,5.08) circle (1.5pt);
\fill  (1.1,5.08) circle (1.5pt);
\fill  (1.18,4.24) circle (1.5pt);
\end{scriptsize}
				\end{tikzpicture}
				\caption{$G_3\in \mathcal{D}\setminus \mathcal{D}_1$ with $H$ is not an induced subgraph}
			\end{figure}
		\end{minipage}
		\begin{minipage}{.5\linewidth}
    \captionsetup[figure]{labelformat=empty}
			
			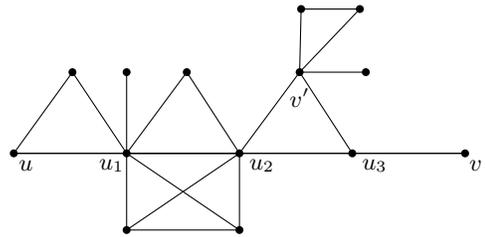
\begin{figure}[H]
				\begin{tikzpicture}[scale=1]
				\draw (-3.5,3.16)-- (-2,3.16);
\draw (-2,3.16)-- (-0.5,3.16);
\draw (-0.5,3.16)-- (1,3.16);
\draw (1,3.16)-- (2.5,3.16);
\draw (-2,3.16)-- (-2,2.14);
\draw (-2,2.14)-- (-0.5,2.14);
\draw (-0.5,3.16)-- (-0.5,2.14);
\draw (-3.5,3.16)-- (-2.72,4.24);
\draw (-2.72,4.24)-- (-2,3.16);
\draw (-2,3.16)-- (-2,4.24);
\draw (-2,3.16)-- (-1.2,4.24);
\draw (-1.2,4.24)-- (-0.5,3.16);
\draw (-0.5,3.16)-- (0.3,4.24);
\draw (0.3,4.24)-- (1.18,4.24);
\draw (0.3,4.24)-- (1.1,5.08);
\draw (1.1,5.08)-- (0.32,5.08);
\draw (0.32,5.08)-- (0.3,4.24);
\draw (-2,3.16)-- (-0.5,2.14);
\draw (-2,2.14)-- (-0.5,3.16);
\draw (-3.5,3.16)-- (2.5,3.16);
\draw (-2,3.16)-- (-0.5,3.16);
\draw (1,3.16)-- (0.3,4.24);
\begin{scriptsize}
\fill (-3.5,3.16) circle (1.5pt);
\draw (-3.34,3) node {$u$};
\fill  (-2,3.16) circle (1.5pt);
\draw (-2.2,3) node {$u_1$};
\fill  (-0.5,3.16) circle (1.5pt);
\draw (-0.2,3) node {$u_2$};
\fill  (1,3.16) circle (1.5pt);
\draw (1.3,3) node {$u_3$};
\fill  (2.5,3.16) circle (1.5pt);
\draw (2.64,3) node {$v$};
\fill  (-2.72,4.24) circle (1.5pt);
\fill  (-2,2.14) circle (1.5pt);
\fill  (-0.5,2.14) circle (1.5pt);
\fill  (-1.2,4.24) circle (1.5pt);
\fill  (-2,4.24) circle (1.5pt);
\fill  (0.3,4.24) circle (1.5pt);
\draw (0.3,3.9) node {$v'$};
\fill  (0.32,5.08) circle (1.5pt);
\fill  (1.1,5.08) circle (1.5pt);
\fill  (1.18,4.24) circle (1.5pt);
\end{scriptsize}
				\end{tikzpicture}
				\caption{$G_4\in \mathcal{D}\setminus \mathcal{D}_1$ with $H$ is not an induced subgraph}
			\end{figure}
			\end{minipage}
	\end{minipage}

\vskip 2mm
We now proceed to compute the depth. Let $G$ be a connected non-complete graph on $[n]$ such that $d(G)+f(G)+1=n+2-\kappa(G)$. It follows from \cite[Theorem 3.5]{MKM-depth} and \cite[Theorems 3.19 and 3.20]{BN17} that $d(G)+f(G)\leq \depth(S/J_G)\leq n+2-\kappa(G)=d(G)+f(G)+1.$

First we study the graphs $G$ with $\kappa(G)=1$. The main tool in our proofs is the use of Ohtani lemma (\Cref{ohtani-lemma}). To help the reader imagine $G_w$, for an internal vertex $w$, we have given pictorial illustration of one such instance in each of the proofs. For example, given below are the graphs $(G_4)_{v'}$ and $G_4 \setminus v'$, where $G_4$ is the graph given above.

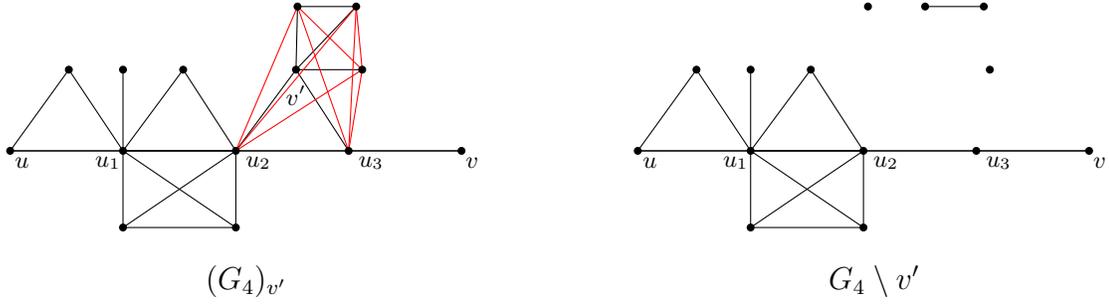
\begin{minipage}{\linewidth}
		\begin{minipage}{.5\linewidth}
		\captionsetup[figure]{labelformat=empty}
		\begin{figure}[H]
				\begin{tikzpicture}[scale=1]
				\draw (-3.5,3.16)-- (-2,3.16);
\draw (-2,3.16)-- (-0.5,3.16);
\draw (-0.5,3.16)-- (1,3.16);
\draw (1,3.16)-- (2.5,3.16);
\draw (-2,3.16)-- (-2,2.14);
\draw (-2,2.14)-- (-0.5,2.14);
\draw (-0.5,3.16)-- (-0.5,2.14);
\draw (-3.5,3.16)-- (-2.72,4.24);
\draw (-2.72,4.24)-- (-2,3.16);
\draw (-2,3.16)-- (-2,4.24);
\draw (-2,3.16)-- (-1.2,4.24);
\draw (-1.2,4.24)-- (-0.5,3.16);
\draw (-0.5,3.16)-- (0.3,4.24);
\draw (0.3,4.24)-- (1.18,4.24);
\draw (0.3,4.24)-- (1.1,5.08);
\draw (1.1,5.08)-- (0.32,5.08);
\draw (0.32,5.08)-- (0.3,4.24);
\draw (-2,3.16)-- (-0.5,2.14);
\draw (-2,2.14)-- (-0.5,3.16);
\draw (-3.5,3.16)-- (2.5,3.16);
\draw (-2,3.16)-- (-0.5,3.16);
\draw (1,3.16)-- (0.3,4.24);
\draw [color=red] (0.32,5.08) -- (1.18,4.24);
\draw [color=red](0.32,5.08) -- (-0.5,3.16);
\draw [color=red](0.32,5.08) -- (1,3.16);
\draw [color=red](1.18,4.24) -- (1.1,5.08);
\draw [color=red](1.18,4.24) -- (-0.5,3.16);
\draw [color=red](1.18,4.24) -- (1,3.16);
\draw [color=red] (1.1,5.08)-- (-0.5,3.16);
\draw [color=red] (1.1,5.08)-- (1,3.16);
\begin{scriptsize}
\fill (-3.5,3.16) circle (1.5pt);
\draw (-3.34,3) node {$u$};
\fill  (-2,3.16) circle (1.5pt);
\draw (-2.2,3) node {$u_1$};
\fill  (-0.5,3.16) circle (1.5pt);
\draw (-0.2,3) node {$u_2$};
\fill  (1,3.16) circle (1.5pt);
\draw (1.3,3) node {$u_3$};
\fill  (2.5,3.16) circle (1.5pt);
\draw (2.64,3) node {$v$};
\fill  (-2.72,4.24) circle (1.5pt);
\fill  (-2,2.14) circle (1.5pt);
\fill  (-0.5,2.14) circle (1.5pt);
\fill  (-1.2,4.24) circle (1.5pt);
\fill  (-2,4.24) circle (1.5pt);
\fill  (0.3,4.24) circle (1.5pt);
\draw (0.3,3.9) node {$v'$};
\fill  (0.32,5.08) circle (1.5pt);
\fill  (1.1,5.08) circle (1.5pt);
\fill  (1.18,4.24) circle (1.5pt);
\end{scriptsize}
				\end{tikzpicture}
				\caption{$(G_4)_{v'}$}
			\end{figure}
		\end{minipage}
		\begin{minipage}{.5\linewidth}
    \captionsetup[figure]{labelformat=empty}
			
			\begin{figure}[H]
				\begin{tikzpicture}[scale=1]
				\draw (-3.5,3.16)-- (-2,3.16);
\draw (-2,3.16)-- (-0.5,3.16);
\draw (-0.5,3.16)-- (1,3.16);
\draw (1,3.16)-- (2.5,3.16);
\draw (-2,3.16)-- (-2,2.14);
\draw (-2,2.14)-- (-0.5,2.14);
\draw (-0.5,3.16)-- (-0.5,2.14);
\draw (-3.5,3.16)-- (-2.72,4.24);
\draw (-2.72,4.24)-- (-2,3.16);
\draw (-2,3.16)-- (-2,4.24);
\draw (-2,3.16)-- (-1.2,4.24);
\draw (-1.2,4.24)-- (-0.5,3.16);
\draw (1.1,5.08)-- (0.32,5.08);
\draw (-2,3.16)-- (-0.5,2.14);
\draw (-2,2.14)-- (-0.5,3.16);
\draw (-3.5,3.16)-- (2.5,3.16);
\draw (-2,3.16)-- (-0.5,3.16);
\begin{scriptsize}
\fill (-3.5,3.16) circle (1.5pt);
\draw (-3.34,3) node {$u$};
\fill  (-2,3.16) circle (1.5pt);
\draw (-2.2,3) node {$u_1$};
\fill  (-0.5,3.16) circle (1.5pt);
\draw (-0.2,3) node {$u_2$};
\fill  (1,3.16) circle (1.5pt);
\draw (1.3,3) node {$u_3$};
\fill  (2.5,3.16) circle (1.5pt);
\draw (2.64,3) node {$v$};
\fill  (-2.72,4.24) circle (1.5pt);
\fill  (-2,2.14) circle (1.5pt);
\fill  (-0.5,2.14) circle (1.5pt);
\fill  (-1.2,4.24) circle (1.5pt);
\fill  (-2,4.24) circle (1.5pt);
\fill  (-0.44,5.08) circle (1.5pt);
\fill  (0.32,5.08) circle (1.5pt);
\fill  (1.1,5.08) circle (1.5pt);
\fill  (1.18,4.24) circle (1.5pt);
\end{scriptsize}
				\end{tikzpicture}
				\caption{$G_4\setminus v'$}
			\end{figure}
			\end{minipage}
	\end{minipage}

\begin{theorem}\label{claim1}
Let $G$ be a chordal graph with $\kappa(G) = 1$ and $d(G)+f(G) = n$. If $G \in \mathcal{D} \setminus \mathcal{D}_1$, then $\depth(S/J_G)=n+1.$
\end{theorem}
\begin{proof} 
Note that $G \in \mathcal{D} \setminus \mathcal{D}_1$ if and only if one of the following conditions hold:
\begin{enumerate}
    \item either $u$ or  $v$ is an internal vertex;
    \item if both $u$ and $v$ are simplicial vertices, then $H$ is not  an induced subgraph of $G$, i.e., $\{v', u_{j+1}\}$ or $\{v', u_{j+2}\}$ is not an edge of $G$;
    \item if both $u$ and $v$ are simplicial vertices and $H$ is an induced subgraph of $G$, then $v'$ or $u_{j+1}$ is a cut-vertex of $G$, i.e., there are cliques attached along $v'$ or $u_{j+1}$.
\end{enumerate}

First we identify an appropriate internal vertex $w$ so that Ohtani's lemma (\Cref{ohtani-lemma}) can be applied to a short exact sequence of the type \eqref{ohtani-ses}. We show that, for a suitable choice of $w$, $\depth(S/J_{G_w}) \geq n+1$, $\depth(S/((x_w,y_w)+J_{G_w\setminus w})) \geq n$ and $\depth(S/((x_w,y_w)+J_{G\setminus w}))\geq n+1$.
We choose $w$ as follows:
\begin{enumerate}
    \item If $u$ is an internal vertex in $G$, then set $w=u_1$ (if $v$ is an internal vertex, then set $w =u_{d-1}$).
    \item Suppose both $u,v$ are simplicial vertices. 
    \begin{enumerate}
    \item Assume that $H$ is not an induced subgraph of $G$, i.e.,  $N_P(v')\subseteq \{u_j,u_{j+1}\}$.
    \begin{enumerate}
        \item If $N_P(v')\subseteq \{u_j,u_{j+1}\}$ and there are cliques attached to $v'$, then set $w=v'$.
    \item If $N_P(v')=\{u_j\}$ and there are no cliques attached to $v'$, then set $w = u_{j}$.
    \item If $N_P(v')= \{u_j,u_{j+1}\}$, and there are no cliques attached to $v'$, then set $w=u_{j+1}$ for $j\leq d-2$. If $j=d-1$ i.e., $u_{j+1}=v$, then choose $w=u_j$. 
    \end{enumerate}
    \item Assume that $H$ is an induced subgraph of $G$. If cliques are attached to $u_{j+1}$, then set $w=u_{j+1}$ otherwise set $w=v'$.
    \end{enumerate}
\end{enumerate}

(1) Since there are no cliques attached along the vertex $u$ and there are no internal vertices in $G$ other than those in $V(P)$, any clique containing $u$ will contain $u_1$ as well, i.e., $N_G[u] \subseteq N_G[u_1].$ Thus, in $G_w$, both $u$ and $u_1=w$ become simplicial vertices. Hence $f(G_w) \geq f(G)+2$. Therefore it follows from \Cref{diameter-lemma} that $n+1 = d(G)+f(G)+1 \leq d(G_w)+f(G_w) \leq n+1.$ This shows that $G_w$ belongs to the Hibi-Madani class. Since $G_w$ is in the Hibi-Madani class and $w$ is a simplicial vertex in $G_w$, $G_w\setminus w$ is also in the Hibi-Madani class. Hence $\depth(S/J_{G_w}) = n+1$ and $\depth(S/((x_w,y_w)+J_{G_w\setminus w}))=n$. Observe that $G \setminus w$ is a disconnected graph. The component containing $u$ is a clique sum of complete graphs along $u$. Hence this component belongs to the Hibi-Madani class. If there are cliques attached to $w$ in $G$, then there will be components in $G\setminus w$ which are cliques, not containing $u$. The component containing $v$ in $G\setminus w$ will be a repeated clique sum of complete graphs along the edges and vertices of the path on the vertices $u_2,\ldots,u_{d-1},v$. Hence this component also belongs to the Hibi-Madani class. Thus for each component $G_i$ with $V(G_i) =n_i$, we get $\depth(S_{G_i}/J_{G_i})=n_i+1$. Therefore, if $c$ denotes the number of components in $G\setminus w$, then $\depth(S/((x_w,y_w)+J_{G\setminus w})) = (n-1)+c \geq n+1.$

(2a(i)) Observe that $G\setminus w$ is a disconnected graph and has no internal vertices other than those in $V(P)$. Cliques attached along $w$ in $G$, will form disconnected components in $G\setminus w$ and the component containing the path $P$ is a clique sum of complete graphs along the vertices and edges of the path $P$. Hence each component of $G\setminus w$ belongs to the Hibi-Madani class. Therefore, $\depth(S/((x_w,y_w)+J_{G\setminus w})) \geq n+1$. Since $N_P(w) \subseteq \{u_j,u_{j+1}\}$, the path $P$ remains an induced path in $G_w$. Moreover, there are no internal vertices in $G_w$ other than those in $V(P)$, this is the only induced path in $G_w$, joining $u$ and $v$. Thus $\dist_{G_w}(u,v) = d$, and hence by \Cref{diameter-lemma}, we get $d(G_w) = d$. By \cite[Lemma 3.2]{Arv-Jaco}, we get $f(G) \leq f(G_w)-1$. Therefore, $n+1 = d(G)+f(G)+1 \leq d(G_w)+f(G_w) \leq n+1$. Hence $G_w$ belongs to the Hibi-Madani class. As we concluded in (1), we can see that $G_w \setminus w$ also belongs to the Hibi-Madani class. Therefore, $\depth(S/J_{G_w}) = n+1$ and $\depth(S/((x_w,y_w)+J_{G_w\setminus w})) = n$.

(2a(ii)) Since there are no cliques attached to $v'$ and $v'$ is an internal vertex, there are more than one clique along the edge $\{v',u_j\}$. Therefore, $N_G[v']\subseteq N_G[w]$ and hence, $v'$ becomes a simplicial vertex in $G_w$. Thus $f(G_w)\geq f(G)+2$ and $d(G)-1\leq d(G_w)$ yields $n+1=d(G)+f(G)+1\leq d(G_w)+f(G_w)=n+1$. So, $G_w$ belongs to the Hibi-Madani class and similarly, one can show that $G_w\setminus w$ also belongs to the Hibi-Madani class. Hence, $\depth(S/J_{G_w}) = n+1$ and $\depth(S/((x_w,y_w)+J_{G_w\setminus w})) = n$. Now $G\setminus w$ is a disconnected graph whose component containing $v'$ is a clique sum of complete graphs along $v'$, component containing $u$ is a repeated clique sum along the vertices and edges of the path $P_1:u,u_1,\ldots,u_{j-1}$, component containing $v$ is a repeated clique sum along the vertices and edges of the path $P_2:u_{j+1},\ldots,u_{d-1},v$ and if there are cliques attached along $u_j$, then we get those cliques  without the vertex $u_j$ as components. Thus each component of $G\setminus w$ belongs to the Hibi-Madani class and hence $\depth(S/((x_w,y_w)+J_{G\setminus w}))\geq n+1$.

(2a(iii)) If $N_G[v']\subseteq N_G[w]$ (for instance, if $j=d-1$), then $v'$ becomes a simplicial vertex in $G_w$ and one can show as in the above case that $G_w,G_w\setminus w$ and each component of $G\setminus w$ is a clique or belong to the Hibi-Madani class. Hence, $\depth(S/J_{G_w}) = n+1, \depth(S/((x_w,y_w)+J_{G_w\setminus w})) = n$ and $\depth(S/((x_w,y_w)+J_{G\setminus w}))\geq n+1$.

If $N_G[v']\nsubseteq N_G[w]$ i.e., $w=u_{j+1}$ for some $j\leq d-2$ and cliques are attached along the edge $\{v',u_j\}$, then $v'$ remains an internal vertex in $G_w$. Moreover, $u_i$ remains an internal vertex for $i \neq j+1$. Thus, $f(G_w)=f(G)+1$. Also since the internal vertices in $G_w$ are $\{u_1,\ldots,u_{j},u_{j+2},\ldots,u_{d-1},v'\}$ and $\{v',u_{j}\},\{v',u_{j+2}\}\in E(G_w)$, $d(G_w) < d$. Hence, by \Cref{diameter-lemma}, $d(G_w)=d(G)-1$. Therefore we can see that $d(G_w)+f(G_w)+1=n+1$. Also, $N_{P'}(v')=\{u_j,u_{j+2}\}$ and $N_{G_w}[v']\subseteq N_{G_w}[u_j],$ where $P':u,u_1,\ldots,u_j,u_{j+2},\ldots,u_{d-1},v$ in $G_w$. Therefore, $G_w$ is the graph discussed in the first paragraph of the proof of (2a(iii)) and hence  $\depth(S/J_{G_w})\geq n+1$ (follows from the depth conclusions in the first paragraph together with \cite[Proposition 1.2.9]{bh} applied to the short exact sequence \eqref{ohtani-ses}). Similarly, one can get $\depth(S/((x_w,y_w)+J_{G_w\setminus w}))\geq n$. As we have seen earlier, $G\setminus w$ is a disconnected graph whose each component belongs to the Hibi-Madani class. So, $\depth(S/((x_w,y_w)+J_{G\setminus w}))\geq n+1$.

(2b) Assume that cliques are attached to $u_{j+1}$. If $N_G[v']\subseteq N_G[w]$, then $v'$ becomes simplicial vertex in $G_w$. Then arguments similar to those in the proof of (2a(ii)), we can show that $\depth(S/J_{G_w}) = n+1$ and $\depth(S/((x_w,y_w)+J_{G_w\setminus w})) = n$. Moreover, $G \setminus w$ is a disconnected graph whose each component is either a clique (coming from a clique attached to $u_{j+1}$) or a repeated clique sum of complete graphs along the path on the vertices $u,u_1,\ldots,u_j, v', u_{j+2},\ldots,u_{d-1},v$. Hence each such component is either a complete graph or belongs to the Hibi-Madani class. Therefore, $\depth(S/((x_w,y_w)+J_{G\setminus w})) \geq n+1$.

If $N_G[v']\nsubseteq N_G[w]$, i.e., cliques are attached along $v', \{v',u_j\}$ or $\{v',u_{j+2}\}$. In this case, $v'$ remains an internal vertex in $G_w$ and  $d(G_w)+f(G_w)+1=n+1$. By taking $P'':u,u_1,\ldots,u_j,u_{j+2},\ldots,u_{d-1},v$, we see that $N_{P''}(v')=\{u_j,u_{j+2}\}$. Therefore, it follows from the proof of (2a(i)) and (2a(iii)) that $\depth(S/J_{G_w})\geq n+1$. Similarly, one can show that $\depth(S/((x_w,y_w)+J_{G_w\setminus w}))\geq n$. In this case also, it is easy to see that $G\setminus w$ is a disconnected graph whose each component is a clique or belongs to the Hibi-Madani class. So, $\depth(S/((x_w,y_w)+J_{G\setminus w}))\geq n+1$.

If cliques are not attached to $u_{j+1}$, then by the definition of the class of graphs $\mathcal{D}\setminus \mathcal{D}_1$, cliques are necessarily attached to $v'$. Taking $w=v'$, we see that the graph is of the type discussed in the first two paragraph of the proof of (2b) with the roles of $u_{j+1}$ and $v'$ interchanged. Hence, arguments similar to those give us the required inequalities on the depth.

In each of the above cases, we have shown that $\depth(S/J_{G_w}) \geq n+1$, $\depth(S/((x_w,y_w)+J_{G_w\setminus w})) \geq n$ and $\depth(S/((x_w,y_w)+J_{G\setminus w})) \geq n+1$. Therefore, by applying \cite[Proposition 1.2.9]{bh} to the short exact sequence \eqref{ohtani-ses}, we get $\depth(S/J_G)\geq n+1$.
\end{proof}

\begin{minipage}{\linewidth}
		\begin{minipage}{.5\linewidth}
		\captionsetup[figure]{labelformat=empty}
		\begin{figure}[H]
				\begin{tikzpicture}[scale=1]
				\draw (-3.5,3.16)-- (-2,3.16);
\draw (-2,3.16)-- (-0.5,3.16);
\draw (-0.5,3.16)-- (1,3.16);
\draw (-2,3.16)-- (-2,4.24);
\draw [color=red](-2,3.16)-- (-0.5,3.16);
\draw (-0.5,4.24)-- (-0.5,3.16);
\draw (-2,3.16)-- (-0.5,4.24);
\draw (-2.4,4.16)--(-2,3.16);
\draw (-3,3.56)--(-2,3.16);
\draw (-2.4,4.16)--(-3,3.56);
\draw (-1.25,4.78)--(-2,3.16);
\draw (-1.25,4.78)--(-0.5,3.16);
\draw (-1.25,4.78)--(-2,4.24);
\draw (-1.25,4.78)--(-0.5,4.24);
\begin{scriptsize}
\fill  (-2.4,4.16) circle (1.5pt);
\fill  (-3,3.56) circle (1.5pt);
\fill  (-3.5,3.16) circle (1.5pt);
\draw (-3.34,3) node {$u$};
\fill  (-2,3.16) circle (1.5pt);
\draw (-1.7,3) node {$u_1$};
\fill  (-0.5,3.16) circle (1.5pt);
\draw (-0.2,3) node {$u_3$};
\fill  (1,3.16) circle (1.5pt);
\draw (1.3,3) node {$v$};
\fill  (-1.25,4.78) circle (1.5pt);
\draw (-1.25,5) node {$v'$};
\fill  (-2,4.24) circle (1.5pt);
\fill  (-0.5,4.24) circle (1.5pt);
\draw (-0.2,4.5) node {$u_2$};
\end{scriptsize}
				\end{tikzpicture}
				\caption{$(G_1)_{u_2}$}
			\end{figure}
		\end{minipage}
		\begin{minipage}{.5\linewidth}
    \captionsetup[figure]{labelformat=empty}
			\begin{figure}[H]
				\begin{tikzpicture}[scale=1]
				\draw (-3.5,3.16)-- (-2,3.16);
\draw (-2,3.16)-- (-0.5,3.16);
\draw (-0.5,3.16)-- (1,3.16);
\draw (1,3.16)-- (2.5,3.16);
\draw (-2,3.16)-- (-0.5,3.16);
\draw (-1.25,4.24)--(-0.5,3.16);
\draw (-1.25,4.24)--(-2,3.16);
\draw (-2.4,4.16)--(-2,3.16);
\draw (-3,3.56)--(-2,3.16);
\draw (-2.4,4.16)--(-3,3.56);
\begin{scriptsize}
\fill  (-2.4,4.16) circle (1.5pt);
\fill  (-3,3.56) circle (1.5pt);
\fill  (-3.5,3.16) circle (1.5pt);
\draw (-3.34,3) node {$u$};
\fill  (-2,3.16) circle (1.5pt);
\draw (-1.7,3) node {$u_1$};
\fill  (-0.5,3.16) circle (1.5pt);
\draw (-0.2,3) node {$v'$};
\fill  (1,3.16) circle (1.5pt);
\draw (1.3,3) node {$u_3$};
\fill  (2.5,3.16) circle (1.5pt);
\draw (2.64,3) node {$v$};
\fill  (-1.25,4.24) circle (1.5pt);
\end{scriptsize}
				\end{tikzpicture}
				\caption{$G_1\setminus u_2$}
			\end{figure}
			
			\end{minipage}
	\end{minipage}

\begin{theorem}\label{claim2}
If $G\in \mathcal{D}_1$, then $\depth(S/J_G)=n$.
\end{theorem}

\begin{proof} 
Since $n=d(G)+f(G) \leq \depth (S/J_G)$, it is enough to prove that $\pd(S/J_G)\geq n.$ We show that $\beta_{n,n+d(G)}(S/J_G)\neq 0$.

Set $w=u_{j+1}$. Since $G$ is clique sum of complete graphs, $u_j$ and $u_{j+2}$ are adjacent to $v'$ and there are no cliques attached along $u_{j+1}$, $G\setminus w$ is a connected graph. By \Cref{g-v}, we get $d(G)\leq d(G\setminus w)$ and $f(G) \leq f(G\setminus w)$. Hence \[n=d(G)+f(G) \leq d(G\setminus w) + f(G\setminus w) \leq (n-1)+2-\kappa(G\setminus w) = n.\]
Since both the above inequalities become equalities, we actually get $d(G) = d(G\setminus w)$, $f(G) = f(G\setminus w)$ and $d(G\setminus w)+f(G\setminus w)=|V(G\setminus w)|+1$. Therefore, by \Cref{extremal-hibi-madani}, $\beta_{n,n+d}(S/((x_w,y_w)+J_{G\setminus w}))$ is an extremal Betti number.

Now we study the graphs $G_w$ and $G_w\setminus w$. If $N_G[v']\subseteq N_G[w]$, then both $w$ and $v'$ become simplicial vertices in $G_w$. Therefore,  $f(G_w)\geq f(G)+2$ and $f(G_w\setminus w)\geq f(G)+1$. If $d(G_w)=d$, then the induced path providing the diameter gives at least $d-1$ internal vertices in $G_w$ which contradicts the fact that $\iv(G_w)\leq d-2$. Thus, \Cref{diameter-lemma} yields that $d(G_w)=d(G)-1$. Similarly, we can see that $d(G_w\setminus w)=d(G)-1$. Now
\[ n+1=d(G)+f(G)+1\leq d(G_w)+f(G_w)\leq n+1.
\]
Hence, $G_w$ belongs to the Hibi-Madani class. Similarly, we can show that $G_w\setminus w$ also belongs to the Hibi-Madani class. Therefore it follows from \Cref{extremal-hibi-madani} and the fact $d(G_w)=d(G_w\setminus w)=d(G)-1$ that
\[\beta_{n-1,n-1+d-1}(S/J_{G_w}) ~ ~ \text{ and } \beta_{n,n+d-1}(S/((x_w,y_w)+J_{G_w\setminus w})) \] 
are extremal Betti numbers.

Suppose $N_G[v']\nsubseteq N_G[w]$, i.e., cliques are attached along $\{v',u_j\}$ or $\{v',u_{j+2}\}$. 
In this case, $v'$ remains an internal vertex in $G_w$ and so $f(G_w)=f(G)+1$. The internal vertices of $G_w$ are $u_1,\ldots,u_j,u_{j+2},\ldots,u_{d-1},v'.$ Since any path of length $d$ has at least $d-1$ internal vertices and $\{v',u_j\},\{v',u_{j+2}\}\in E(G_w)$, there is no induced path of length $d$ in $G_w$. Thus $d(G_w)\leq d-1$ and hence by \Cref{diameter-lemma}, $d(G_w)=d-1$. Therefore, $d(G_w)+f(G_w)+1=n+1$. Taking $P': u,u_1,\ldots,u_j,u_{j+2},\ldots,u_{d-1},v$ in $G_w$ and observing that $v'$ is an internal vertex in $G_w$ with $N_{P'}(v')=\{u_j,u_{j+2}\}$, we can see that $G_w$ is a graph of the type considered in \Cref{claim1}(2a(iii)). It follows from \Cref{claim1} and Auslander-Buchsbaum formula that $\pd(S/J_{G_w})=n-1$.

\noindent
\textbf{Claim:} $\beta_{n-1,n-1+d-1}(S/J_{G_w})$ is an extremal Betti number of $S/J_{G_w}$.

\noindent
\textit{Proof of the Claim}: Let $H=G_w$. Choose $w'=u_{j+2}$ if $0\leq j\leq d-3$, otherwise choose $w'=u_j$. If $N_H[v']\subseteq N_H[w']$ (for instance, if $j=d-2$), then it follows from the proof of \Cref{claim1}(2a(iii)) that $H_{w'},H_{w'}\setminus w'$ and each component of the disconnected graph $H\setminus w'$ is a clique or belong to the Hibi-Madani class with $d(H_{w'})=d(H_{w'}\setminus w')=d(H)-1=d-2$. Hence it follows from \Cref{extremal-hibi-madani} and \eqref{Bettiproduct} that  $$\beta_{n-1,n-1+d-2}(S/J_{H_{w'}}) \text{ and }\beta_{n,n+d-2}(S/((x_{w'},y_{w'})+J_{H_{w'}\setminus w'}))  $$ are extremal Betti numbers. Since $H\setminus w'$ is disconnected, by \cite[Corollary 2.5]{Hibi-Madani}, Auslander-Buchsbaum formula and \eqref{Bettiproduct}, $\pd(S/((x_{w'},y_{w'})+J_{H\setminus w'}))\leq n-1$  and as in the proof of \Cref{extremal-hibi-madani}, we can show that $\beta_{n-1,n-1+j}(S/((x_{w'},y_{w'})+J_{H\setminus w'}))=0$ for $j\geq d$. Therefore, it follows from \eqref{ohtani-tor} that for $j\geq d-2$
\begin{equation*}
\Tor_{n}^{S}\left( \frac{S}{(x_{w'},y_{w'})+J_{H_{w'}\setminus w'}},\K\right)_{n+j} \simeq \Tor_{n-1}^{S}\left( \frac{S}{J_{H}},\K\right)_{n+j}.
\end{equation*}
This implies that $\beta_{n-1,n+d-2}(S/J_H)$ is an extremal Betti number of $S/J_H$.

Suppose $N_H[v']\nsubseteq N_H[w']$, i.e., cliques are attached along $\{v',u_j\}$. Then it follows from the proof of \Cref{claim1}(2a(iii)) that $H_{w'}$ and $H_{w'}\setminus w'$ belong to the class of graphs discussed in the previous paragraph and hence $\beta_{n-1,n-1+d-2}(S/J_{H_{w'}})$ and $\beta_{n,n+d-2}(S/((x_{w'},y_{w'})+J_{H_{w'}\setminus w'}))$ are extremal Betti numbers. Also, $H\setminus w'$ is a disconnected graph whose each component is a clique or belongs to the Hibi-Madani class. Thus, it follows from \cite[Corollary 2.5]{Hibi-Madani}, Auslander-Buchsbaum formula and \eqref{Bettiproduct} that $\pd(S/((x_{w'},y_{w'})+J_{H\setminus w'}))\leq n-1.$ Therefore, from long exact sequence \eqref{ohtani-tor}, we see that $\beta_{n-1,n-1+d-1}(S/J_H)$ is an extremal Betti number of $S/J_H$. This completes the proof of the claim.

Structurally, $G_w$ and $G_w\setminus w$ are very similar. Using arguments similar to the ones used above, we can show that $\beta_{n,n+d-1}(S/((x_w,y_w)+J_{G_w\setminus w}))$ is an extremal Betti number of $S/((x_w,y_w)+J_{G_w\setminus w})$.

Taking $i=n$ and $j=d$ in \eqref{ohtani-tor}, we get
\begin{equation*}
\Tor_{n}^{S}\left( \frac{S}{J_G},\K\right)_{n+d}\simeq \Tor_{n}^{S}\left( \frac{S}{(x_w,y_w)+J_{G\setminus w}},\K\right)_{n+d}.
\end{equation*}
This implies that $\beta_{n,n+d}(S/J_G)\neq 0$ so that $\pd(S/J_G) \geq n.$ Hence, $\depth(S/J_G)=n.$
\end{proof}

Now we consider non-chordal graphs, i.e., graphs having precisely one cycle, $C_4$, as an induced subgraph. Since $\kappa(G)=1$, there is only one internal vertex outside $V(P)$. Note also that all vertices of $C_4$ are internal. Hence $|V(C_4) \cap V(P)|=3$, say $V(C_4) \cap V(P) = \{u_j, u_{j+1}, u_{j+2}\}$ for some $j\in \{1, \ldots, d-3\}$. Consider the graph $H'$:

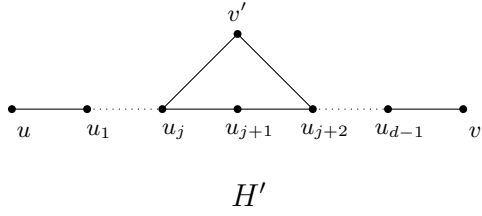
\captionsetup[figure]{labelformat=empty}
			
			\begin{figure}[H]
\begin{tikzpicture}[scale=1]
\draw[dotted] (-6, 0) -- (-5,0);
\draw[dotted] (-3, 0) -- (-2,0);
\draw (-7,0)-- (-6,0);
\draw  (-5,0)-- (-4,0);
\draw  (-4,0)-- (-3,0);
\draw  (-4,1)-- (-5,0);
\draw  (-4,1)-- (-3,0);
\draw  (-2,0)-- (-1,0);
\begin{scriptsize}
\fill  (-7,0) circle (1.5pt);
\draw (-6.84,-0.3) node {$u$};
\fill  (-6,0) circle (1.5pt);
\draw (-5.84,-0.3) node {$u_1$};
\fill  (-5,0) circle (1.5pt);
\draw (-4.84,-0.3) node {$u_j$};
\fill  (-4,0) circle (1.5pt);
\draw (-3.84,-0.3) node {$u_{j+1}$};
\fill  (-3,0) circle (1.5pt);
\draw (-2.84,-0.3) node {$u_{j+2}$};
\fill (-2,0) circle (1.5pt);
\draw (-1.84,-0.3) node {$u_{d-1}$};
\fill  (-1,0) circle (1.5pt);
\draw (-0.84,-0.3) node {$v$};
\fill  (-4,1) circle (1.5pt);
\draw (-4,1.3) node {$v'$};
\end{scriptsize}
				\end{tikzpicture}
				\caption{$H'$}
			\end{figure}

Note that $H'$ is an induced subgraph of $G$ in this situation.
Let $C: u_j, u_{j+1}, u_{j+2}, v', u_j$ be the induced $4$-cycle in $G$, where $v'  \notin V(P)$. Then $G \setminus v'$ is a chordal graph having internal vertices $\{u_1,\ldots,u_{d-1}\}$. Hence, $G\setminus v'$ is a clique sum of complete graphs along the complete subgraphs of $G[u_1,\ldots,u_{d-1}]$. It  may also be noted that there are no cliques of size at least 3 in the induced subgraph $G[u_1,\ldots,u_{d-1},v']$. Therefore, in $G$, intersection of any two cliques will either be a vertex or an edge of $G[u_1,\ldots,u_{d-1},v']$.

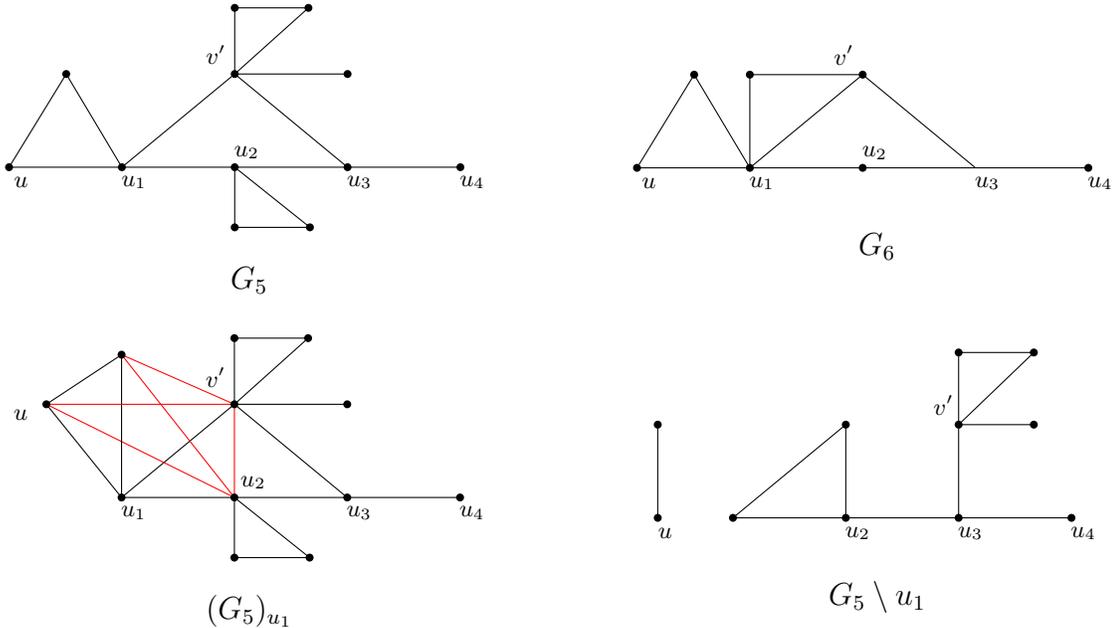
\begin{minipage}{\linewidth}
		\begin{minipage}{.5\linewidth}
		
		\captionsetup[figure]{labelformat=empty}
\begin{figure}[H]
				\begin{tikzpicture}[scale=1]
\draw (-3,2.6)-- (-1.5,2.6);
\draw (-1.5,2.6)-- (0,2.6);
\draw (0,2.6)-- (1.5,2.6);
\draw (1.5,2.6)-- (3,2.6);
\draw (-1.5,2.6)-- (0,3.84);
\draw (0,3.84)-- (1.5,2.6);
\draw (-3,2.6)-- (-2.24,3.84);
\draw (-2.24,3.84)-- (-1.5,2.6);
\draw (1.5,3.84)-- (0,3.84);
\draw (0,4.72)-- (0.98,4.72);
\draw (0,3.84)-- (0,4.72);
\draw (0,3.84)-- (0.98,4.72);
\draw (0,2.6)-- (0,1.8);
\draw (0,2.6)-- (1,1.8);
\draw (0,1.8)-- (1,1.8);
\begin{scriptsize}
\fill  (-3,2.6) circle (1.5pt);
\draw (-2.84,2.4) node {$u$};
\fill  (-1.5,2.6) circle (1.5pt);
\draw (-1.34,2.4) node {$u_1$};
\fill  (0,2.6) circle (1.5pt);
\draw (0.16,2.8) node {$u_2$};
\fill  (0,1.8) circle (1.5pt);
\fill  (1,1.8) circle (1.5pt);
\fill  (1.5,2.6) circle (1.5pt);
\draw (1.66,2.4) node {$u_3$};
\fill  (3,2.6) circle (1.5pt);
\draw (3.16,2.4) node {$u_4$};
\fill  (0,3.84) circle (1.5pt);
\draw (-0.25,4.1) node {$v'$};
\fill (-2.24,3.84) circle (1.5pt);
\fill  (1.5,3.84) circle (1.5pt);
\fill  (0,4.72) circle (1.5pt);
\fill  (0.98,4.72) circle (1.5pt);
\end{scriptsize}
\end{tikzpicture}
\caption{$G_5$}
\end{figure}
		\end{minipage}
		\begin{minipage}{.5\linewidth}
		\captionsetup[figure]{labelformat=empty}
\begin{figure}[H]
				\begin{tikzpicture}[scale=1]
\draw (-3,2.6)-- (-1.5,2.6);
\draw (-1.5,2.6)-- (0,2.6);
\draw (0,2.6)-- (1.5,2.6);
\draw (1.5,2.6)-- (3,2.6);
\draw (-1.5,2.6)-- (0,3.84);
\draw (0,3.84)-- (1.5,2.6);
\draw (-3,2.6)-- (-2.24,3.84);
\draw (-2.24,3.84)-- (-1.5,2.6);
\draw (-1.5,3.84)-- (-1.5,2.6);
\draw (0,3.84)-- (-1.5,3.84);
\begin{scriptsize}
\fill  (-3,2.6) circle (1.5pt);
\draw (-2.84,2.4) node {$u$};
\fill  (-1.5,2.6) circle (1.5pt);
\draw (-1.34,2.4) node {$u_1$};
\fill  (-1.5,3.84) circle (1.5pt);
\fill  (0,2.6) circle (1.5pt);
\draw (0.16,2.8) node {$u_2$};
\draw (1.66,2.4) node {$u_3$};
\fill  (3,2.6) circle (1.5pt);
\draw (3.16,2.4) node {$u_4$};
\fill  (0,3.84) circle (1.5pt);
\draw (-0.25,4.1) node {$v'$};
\fill (-2.24,3.84) circle (1.5pt);
\end{scriptsize}
\end{tikzpicture}
\caption{$G_6$}
\end{figure}
		\end{minipage}
	\end{minipage}

\begin{minipage}{\linewidth}
		\begin{minipage}{.5\linewidth}
		
		\captionsetup[figure]{labelformat=empty}
\begin{figure}[H]
				\begin{tikzpicture}[scale=1]
\draw (-1.5,2.6)-- (0,2.6);
\draw (0,2.6)-- (1.5,2.6);
\draw (1.5,2.6)-- (3,2.6);
\draw (-1.5,2.6)-- (0,3.84);
\draw (0,3.84)-- (1.5,2.6);
\draw (1.5,3.84)-- (0,3.84);
\draw (0,4.72)-- (0.98,4.72);
\draw (0,3.84)-- (0,4.72);
\draw (0,3.84)-- (0.98,4.72);
\draw (0,2.6)-- (0,1.8);
\draw (0,2.6)-- (1,1.8);
\draw (0,1.8)-- (1,1.8);
\draw (-2.5,3.84)-- (-1.5,4.5);
\draw (-2.5,3.84)-- (-1.5,2.6);
\draw [color=red](0,3.84)-- (-1.5,4.5);
\draw [color=red](0,3.84)-- (0,2.6);
\draw [color=red](-2.5,3.84)-- (0,2.6);
\draw [color=red](-1.5,4.5)-- (0,2.6);
\draw (-1.5,4.5)-- (-1.5,2.6);
\draw [color=red](0,3.84)-- (-2.5,3.84);
\begin{scriptsize}
\fill  (-2.5,3.84) circle (1.5pt);
\draw (-2.84,3.7) node {$u$};
\fill  (-1.5,2.6) circle (1.5pt);
\draw (-1.34,2.4) node {$u_1$};
\fill  (0,2.6) circle (1.5pt);
\draw (0.25,2.8) node {$u_2$};
\fill  (0,1.8) circle (1.5pt);
\fill  (1,1.8) circle (1.5pt);
\fill  (1.5,2.6) circle (1.5pt);
\draw (1.66,2.4) node {$u_3$};
\fill  (3,2.6) circle (1.5pt);
\draw (3.16,2.4) node {$u_4$};
\fill  (0,3.84) circle (1.5pt);
\draw (-0.25,4.2) node {$v'$};
\fill (-1.5,4.5) circle (1.5pt);
\fill  (1.5,3.84) circle (1.5pt);
\fill  (0,4.72) circle (1.5pt);
\fill  (0.98,4.72) circle (1.5pt);
\end{scriptsize}
\end{tikzpicture}
\caption{$(G_5)_{u_1}$}
\end{figure}
		\end{minipage}
		\begin{minipage}{.5\linewidth}
		\captionsetup[figure]{labelformat=empty}
\begin{figure}[H]
				\begin{tikzpicture}[scale=1]
\draw (0,2.6)-- (1.5,2.6);
\draw (1.5,2.6)-- (3,2.6);
\draw (0,3.84)-- (0,2.6);
\draw (0,3.84)-- (-1.5,2.6);
\draw (-2.5,2.6)-- (-2.5,3.84);
\draw (0,2.6)--(-1.5,2.6);
\draw (2.5,3.84)-- (1.5,3.84);
\draw (1.5,3.84)-- (1.5,2.6);
\draw (1.5,4.8)-- (2.5,4.8);
\draw (1.5,4.8) --(1.5,3.84);
\draw (2.5,4.8) -- (1.5,3.84);
\begin{scriptsize}
\fill (-2.5,3.84) circle (1.5pt);
\fill  (-2.5,2.6) circle (1.5pt);
\draw (-2.4,2.4) node {$u$};
\fill  (-1.5,2.6) circle (1.5pt);
\fill  (0,2.6) circle (1.5pt);
\draw (0.16,2.4) node {$u_2$};
\fill  (2.5,4.8) circle (1.5pt);
\fill  (1.5,4.8) circle (1.5pt);
\fill  (1.5,2.6) circle (1.5pt);
\draw (1.66,2.4) node {$u_3$};
\fill  (2.5,3.84) circle (1.5pt);
\fill  (3,2.6) circle (1.5pt);
\draw (3.16,2.4) node {$u_4$};
\fill  (1.5,3.84) circle (1.5pt);
\fill  (0,3.84) circle (1.5pt);
\draw (1.3,4.1) node {$v'$};
\end{scriptsize}
\end{tikzpicture}
\caption{$G_5\setminus u_1$}
\end{figure}
		\end{minipage}
	\end{minipage}

We now compute the depth of $S/J_G$ for the above described class of graphs.

\begin{theorem}\label{claim3}
Let $G$ be a non-chordal graph with $\kappa(G)=1$ and $H'$ as an induced subgraph. If $\mathcal{C}(G) \cap \{\{u_{j+1}\},\{ v'\}\} \neq \emptyset$, then $\depth(S/J_G)=n+1.$
\end{theorem}

\begin{proof}
Set $w=u_{j},$ and in $G_w$ set $P':u,u_1,\ldots,u_{j-1},u_{j+1},\ldots,u_{d-1},v$. Then we see that $v'$ is an internal vertex in $G_w$ with $N_{P'}(v')=\{u_{j-1},u_{j+1},u_{j+2}\}$ and cliques are attached to $u_{j+1}$ or to $v'$. Since $d(G_w) < d$, by \Cref{diameter-lemma}, we get $d(G_w)=d-1$. Therefore, $d(G_w)+f(G_w)+1=d(G)+f(G)+1=n+1$ which implies that $G_w, G_w\setminus w\in \in \mathcal{D} \setminus \mathcal{D}_1$. Hence, by \Cref{claim1}, $\depth(S/J_{G_w})=n+1$ and $\depth(S/((x_w,y_w)+J_{G_w \setminus w}))=n$. Also, $G\setminus w$ is a disconnected graph whose each component is either a complete graph or belongs to the Hibi-Madani class (if either $u_{j+1}$ or $v'$ is a simplicial vertex in $G \setminus w$) or belongs to the class $ \mathcal{D} \setminus \mathcal{D}_1$ (if both $u_{j+1}$ and $v'$ are internal vertices in $G \setminus w$). In either case, $\depth(S/((x_w,y_w)+J_{G\setminus w}))\geq n+1$. Hence, by applying \cite[Proposition 1.2.9]{bh} to the short exact sequence $\eqref{ohtani-ses}$, we get $\depth(S/J_G)\geq n+1$. Thus $\depth(S/J_G)=n+1$.
\end{proof}

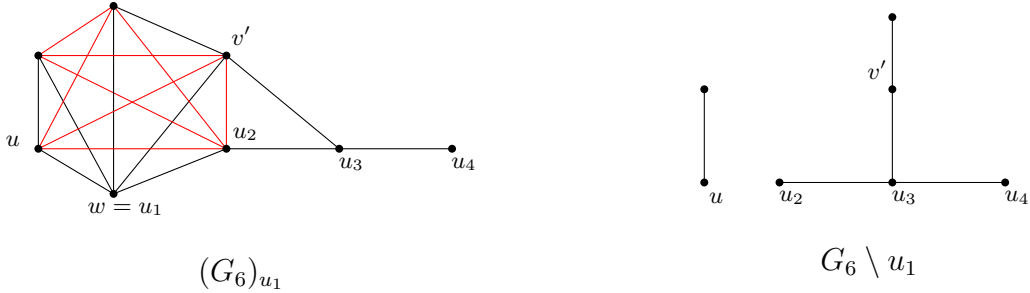
\begin{minipage}{\linewidth}
		\begin{minipage}{.5\linewidth}
		
		\captionsetup[figure]{labelformat=empty}
\begin{figure}[H]
				\begin{tikzpicture}[scale=1]
\draw (-1.5,2)-- (0,2.6);
\draw (0,2.6)-- (1.5,2.6);
\draw (1.5,2.6)-- (3,2.6);
\draw (-1.5,2)-- (0,3.84);
\draw (0,3.84)-- (1.5,2.6);
\draw [color=red](-2.5,3.84)-- (-1.5,4.5);
\draw (-2.5,3.84)-- (-1.5,2);
\draw (0,3.84)-- (-1.5,4.5);
\draw [color=red](0,3.84)-- (0,2.6);
\draw [color=red](-2.5,3.84)-- (0,2.6);
\draw [color=red](-1.5,4.5)-- (0,2.6);
\draw (-1.5,4.5)-- (-1.5,2);
\draw [color=red](0,3.84)-- (-2.5,3.84);
\draw (-1.5,2)--(-2.5,2.6);
\draw [color=red](0,2.6)--(-2.5,2.6);
\draw (-2.5,3.84)--(-2.5,2.6);
\draw [color=red](0,3.84)--(-2.5,2.6);
\draw [color=red](-1.5,4.5)--(-2.5,2.6);
\begin{scriptsize}
\fill  (-2.5,2.6) circle (1.5pt);
\draw (-2.84,2.7) node {$u$};
\fill  (-1.5,2) circle (1.5pt);
\draw (-1.34,1.8) node {$w=u_1$};
\fill  (0,2.6) circle (1.5pt);
\draw (0.25,2.8) node {$u_2$};
\fill  (1.5,2.6) circle (1.5pt);
\draw (1.66,2.4) node {$u_3$};
\fill  (3,2.6) circle (1.5pt);
\draw (3.16,2.4) node {$u_4$};
\fill  (0,3.84) circle (1.5pt);
\draw (0.2,4.1) node {$v'$};
\fill  (-2.5,3.84) circle (1.5pt);
\fill (-1.5,4.5) circle (1.5pt);
\end{scriptsize}
\end{tikzpicture}
\caption{$(G_6)_{u_1}$}
\end{figure}
		\end{minipage}
		\begin{minipage}{.5\linewidth}
		\captionsetup[figure]{labelformat=empty}
\begin{figure}[H]
				\begin{tikzpicture}[scale=1]
\draw (0,2.6)-- (1.5,2.6);
\draw (1.5,2.6)-- (3,2.6);
\draw (-1,2.6)-- (-1,3.84);
\draw (1.5,3.84)-- (1.5,2.6);
\draw (1.5,4.8) --(1.5,3.84);
\begin{scriptsize}
\fill (-1,3.84) circle (1.5pt);
\fill  (-1,2.6) circle (1.5pt);
\draw (-0.84,2.4) node {$u$};
\fill  (0,2.6) circle (1.5pt);
\draw (0.16,2.4) node {$u_2$};
\fill  (1.5,4.8) circle (1.5pt);
\fill  (1.5,2.6) circle (1.5pt);
\draw (1.66,2.4) node {$u_3$};
\fill  (3,2.6) circle (1.5pt);
\draw (3.16,2.4) node {$u_4$};
\fill  (1.5,3.84) circle (1.5pt);
\draw (1.3,4.1) node {$v'$};
\end{scriptsize}
\end{tikzpicture}
\caption{$G_6\setminus u_1$}
\end{figure}
		\end{minipage}
	\end{minipage}

\begin{theorem}\label{claim4}
Let $G$ be a non-chordal graph with $\kappa(G)=1$ and $H'$ as an induced subgraph. If $\mathcal{C}(G) \cap \{\{u_{j+1}\},\{ v'\}\} = \emptyset$, then $\depth(S/J_G)=n.$

\end{theorem}

\begin{proof}
As in the proof of \Cref{claim2}, we show that $\beta_{n,n+d(G)}(S/J_G)\neq 0.$ Taking $w=u_{j}$, we can see that $G_w,G_w\setminus w \in \mathcal{D}_1$. Therefore, it follows from the proof of \Cref{claim2} that $\beta_{n,n+d(G)-1}(S/J_{G_w})\neq 0$ and $\beta_{n+1,n+d(G)}(S/((x_w,y_w)+J_{G_w\setminus w}))\neq 0$ as $d(G)=d(G_w)+1=d(G_w\setminus w)+1$. Also, $G\setminus w$ is a disconnected graph such that each component is  either a complete graph or belongs to the Hibi-Madani class (if either $u_{j+1}$ or $v'$ is a simplicial vertex in $G \setminus w$) or belongs to the class $\mathcal{D} \setminus \mathcal{D}_1$ (if both $u_{j+1}$ and $v'$ are internal vertices in $G \setminus w$). Hence $\depth(S/((x_w,y_w)+J_{G\setminus w})) \geq n+1$ which implies that $\pd(S/((x_w,y_w)+J_{G\setminus w}))\leq n-1$. Therefore, it follows from  \eqref{ohtani-tor} that
$\beta_{n,n+d(G)}(S/J_G)\neq 0$.
\end{proof}
Now we assume that $\kappa(G) \geq 2$. By \Cref{charac-thm}, $2\leq d(G) \leq 3.$ First we assume that $d(G) = 2$. By \Cref{induced-cycle-charac}, $G$ is chordal. Moreover, any two vertices of $G$ are either adjacent or have a common neighbor. Note that $\iv(G) = \kappa+1$. Since $G$ is chordal with $\kappa(G) = \kappa$, it is a repeated clique sum of complete graphs along complete subgraphs of size $\kappa$ or ${\kappa+1}$.  Observe that any maximal clique in $G$ has at least $\kappa+1$ vertices, and $\kappa$ of those vertices must be internal vertices.




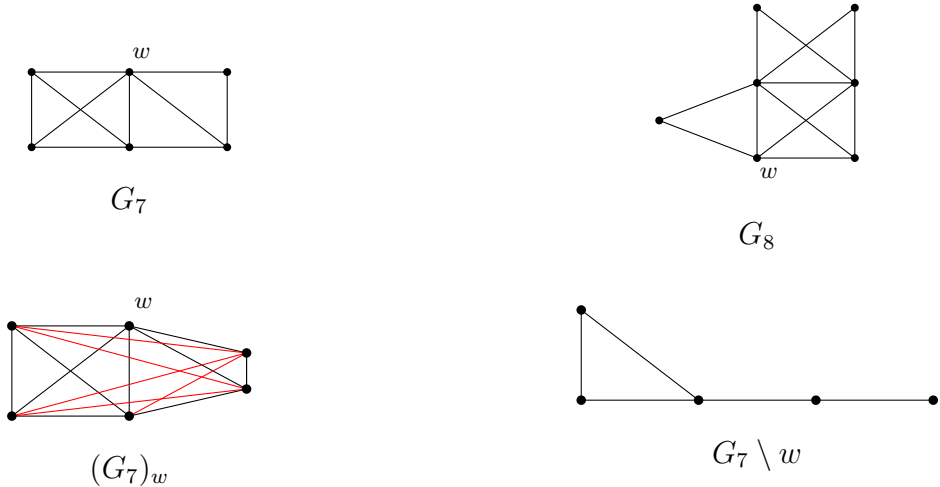
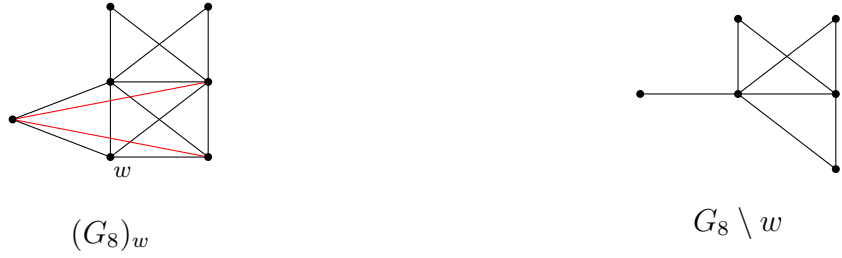
\begin{minipage}{\linewidth}
		\begin{minipage}{.5\linewidth}
		
		\captionsetup[figure]{labelformat=empty}
\begin{figure}[H]
				\begin{tikzpicture}[scale=1]
\draw (0,1)-- (1.3,1);
\draw (1.3,1)-- (1.3,0);
\draw (0,0)-- (0,1);
\draw (0,0)-- (1.3,0);
\draw (1.3,0)-- (0,1);
\draw (0,0)-- (1.3,1);
\draw (1.3,1)-- (2.6,1);
\draw (2.6,0)-- (1.3,0);
\draw (2.6,1)-- (2.6,0);
\draw (1.3,1)-- (2.6,0);
\begin{scriptsize}
\fill  (0,0) circle (1.5pt);
\fill  (1.3,0) circle (1.5pt);
\fill  (2.6,0) circle (1.5pt);
\fill  (1.3,1) circle (1.5pt);
\draw (1.46,1.26) node {$w$};
\fill  (0,1) circle (1.5pt);
\fill (2.6,1) circle (1.5pt);
\end{scriptsize}
\end{tikzpicture}
\caption{$G_7$}
\end{figure}
		\end{minipage}
		\begin{minipage}{.5\linewidth}
		\captionsetup[figure]{labelformat=empty}
\begin{figure}[H]
				\begin{tikzpicture}[scale=1]
\draw (1.3,1)-- (2.6,0);
\draw (2.6,1)-- (2.6,0);
\draw (1.3,0)-- (1.3,1);
\draw (1.3,1)-- (2.6,1);
\draw (1.3,0)-- (2.6,0);
\draw (1.3,0)-- (2.6,1);
\draw (1.3,2)-- (1.3,1);
\draw (1.3,2)-- (2.6,1);
\draw (2.6,2)-- (2.6,1);
\draw (2.6,2)-- (1.3,1);
\draw (0,0.5)-- (1.3,1);
\draw (0,0.5)-- (1.3,0);
\begin{scriptsize}
\fill  (1.3,0) circle (1.5pt);
\draw (1.46,-0.2) node {$w$};
\fill  (2.6,0) circle (1.5pt);
\fill  (1.3,1) circle (1.5pt);
\fill (2.6,1) circle (1.5pt);
\fill  (1.3,2) circle (1.5pt);
\fill  (2.6,2) circle (1.5pt);
\fill  (0,0.5) circle (1.5pt);
\end{scriptsize}
\end{tikzpicture}
\caption{$G_8$}
\end{figure}
		\end{minipage}
	\end{minipage}

	\begin{minipage}{\linewidth}
		\begin{minipage}{.5\linewidth}
		
		\captionsetup[figure]{labelformat=empty}
\begin{figure}[H]
				\begin{tikzpicture}[scale=1.2]
\draw (0,1)-- (1.3,1);
\draw (1.3,1)-- (1.3,0);
\draw (0,0)-- (0,1);
\draw (0,0)-- (1.3,0);
\draw (1.3,0)-- (0,1);
\draw (0,0)-- (1.3,1);
\draw (1.3,1)-- (2.6,0.7);
\draw (2.6,0.3)-- (1.3,0);
\draw (2.6,0.7)-- (2.6,0.3);
\draw (1.3,1)-- (2.6,0.3);
\draw (2.6,0.3)-- (2.6,0.7);
\draw [color=red](0,1)-- (2.6,0.7);
\draw [color=red](2.6,0.7)-- (0,0);
\draw [color=red](1.3,0)-- (2.6,0.7);
\draw [color=red](2.6,0.3)-- (0,1);
\draw [color=red](2.6,0.3)-- (0,0);
\begin{scriptsize}
\fill  (0,0) circle (1.5pt);
\fill  (1.3,0) circle (1.5pt);
\fill  (2.6,0.3) circle (1.5pt);
\fill  (1.3,1) circle (1.5pt);
\draw (1.46,1.26) node {$w$};
\fill  (0,1) circle (1.5pt);
\fill  (2.6,0.7) circle (1.5pt);
\end{scriptsize}
\end{tikzpicture}
\caption{$(G_7)_{w}$}
\end{figure}
		\end{minipage}
		\begin{minipage}{.5\linewidth}
		\captionsetup[figure]{labelformat=empty}
\begin{figure}[H]
				\begin{tikzpicture}[scale=1.2]
\draw (0,1)-- (0,0);
\draw (0,1)-- (1.3,0);
\draw (0,0)-- (1.3,0);
\draw (1.3,0)-- (2.6,0);
\draw (2.6,0)-- (3.9,0);
\begin{scriptsize}
\fill  (0,0) circle (1.5pt);
\fill  (1.3,0) circle (1.5pt);
\fill  (2.6,0) circle (1.5pt);
\fill  (0,1) circle (1.5pt);
\fill (3.9,0) circle (1.5pt);
\end{scriptsize}
\end{tikzpicture}
\caption{$G_7\setminus w$}
\end{figure}
		\end{minipage}
	\end{minipage}

	\begin{minipage}{\linewidth}
		\begin{minipage}{.5\linewidth}
		
		\captionsetup[figure]{labelformat=empty}
\begin{figure}[H]
				\begin{tikzpicture}[scale=1]
\draw (1.3,1)-- (2.6,0);
\draw (2.6,1)-- (2.6,0);
\draw (1.3,0)-- (1.3,1);
\draw (1.3,1)-- (2.6,1);
\draw (1.3,0)-- (2.6,0);
\draw (1.3,0)-- (2.6,1);
\draw (1.3,2)-- (1.3,1);
\draw (1.3,2)-- (2.6,1);
\draw (2.6,2)-- (2.6,1);
\draw (2.6,2)-- (1.3,1);
\draw (0,0.5)-- (1.3,1);
\draw (0,0.5)-- (1.3,0);
\draw [color=red](0,0.5)-- (2.6,1);
\draw [color=red](0,0.5)-- (2.6,0);
\begin{scriptsize}
\fill  (1.3,0) circle (1.5pt);
\draw (1.46,-0.2) node {$w$};
\fill  (2.6,0) circle (1.5pt);
\fill  (1.3,1) circle (1.5pt);
\fill (2.6,1) circle (1.5pt);
\fill  (1.3,2) circle (1.5pt);
\fill  (2.6,2) circle (1.5pt);
\fill  (0,0.5) circle (1.5pt);
\end{scriptsize}
\end{tikzpicture}
\caption{$(G_8)_{w}$}
\end{figure}
		\end{minipage}
		\begin{minipage}{.5\linewidth}
		\captionsetup[figure]{labelformat=empty}
\begin{figure}[H]
				\begin{tikzpicture}[scale=1]
\draw (1.3,1)-- (2.6,0);
\draw (2.6,1)-- (2.6,0);
\draw (1.3,1)-- (2.6,1);
\draw (1.3,2)-- (1.3,1);
\draw (1.3,2)-- (2.6,1);
\draw (2.6,2)-- (2.6,1);
\draw (2.6,2)-- (1.3,1);
\draw (0,1)-- (1.3,1);
\begin{scriptsize}
\fill  (2.6,0) circle (1.5pt);
\fill (1.3,1) circle (1.5pt);
\fill (2.6,1) circle (1.5pt);
\fill  (1.3,2) circle (1.5pt);
\fill  (2.6,2) circle (1.5pt);
\fill (0,1) circle (1.5pt);
\end{scriptsize}
\end{tikzpicture}
\caption{$G_8\setminus w$}
\end{figure}
		\end{minipage}
	\end{minipage}

\vskip 2mm \noindent
\begin{theorem}\label{claim5}
Let $G$ be a graph with $d(G)=2$ and $\kappa(G) \geq 2$. Then
$\depth(S/J_G)=n+2-\kappa.$
\end{theorem}
\begin{proof}
It is enough to prove that $\depth(S/J_G)\geq n+2-\kappa$. Let $w$ be an internal vertex. 
If $w$ belongs to the intersection of all maximal cliques in $G$, then $J_{G_w}$ and $J_{G_w \setminus w}$ are complete graphs.  Hence  $\depth(S/J_{G_w})=n+1$ and $\depth(S/((x_w,y_w)+J_{G_w\setminus w}))=n$.
If there exists a maximal clique in $G$ not containing $w$, then $G_w$ is a clique sum of complete graphs along a fixed $K_\kappa$. Hence $G_w$ belongs to the Hibi-Madani class $\mathcal{F}_{\kappa}$ defined in \cite{Hibi-Madani}. Thus $G_w \setminus w$ is also in $\mathcal{F}_{\kappa}$.
Hence by \cite[Corollary 2.5]{Hibi-Madani}, $\depth(S/J_{G_w}) = n+2-\kappa$ and $\depth(S/((x_w,y_w)+J_{G_w\setminus w})) = n+1-\kappa.$

We now study $G\setminus w$.  Choose the vertex $w$ such that $G\setminus \{w=w_1,\ldots,w_\kappa\}$ is disconnected. We prove the assertion by induction on $\kappa$. Let $\kappa=2$. Hence $\iv(G) = 3$. Hence $G$ is obtained by taking clique sum of complete graphs along a fixed $K_3$ or its edges. Therefore, $G\setminus w$ is a clique sum of complete graphs along a fixed $K_2$ or vertices of the $K_2$. Since $\iv(G) = 3$ and $G\setminus w$ is not a complete graph, $1\leq \iv(G \setminus w) \leq  2$ and $d(G\setminus w) > 1$. We now observe that $d(G\setminus w) \leq 3$. Suppose $d(G\setminus w) \geq 4$, then there exists an induced path of length at least $4$. This implies that $\iv(G\setminus w) \geq 3$, which is a contradiction. Therefore, $2 \leq d(G\setminus w) \leq 3$. Suppose $\iv(G\setminus w)=2$ and $d(G\setminus w)=2$. Then $d(G\setminus w) + f(G \setminus w) + 1 = n = (n-1) + 2 - \kappa(G \setminus w)$. If $G\setminus w \in  \mathcal{D}_1$, then, following the notation set in the definition, $j=0$ with $u_{j+2} = v$. Moreover, $\kappa(G\setminus w) = 1$ and there are no cliques attached either to $v'$ or $u_{j+1}$, which are the only internal vertices. This contradicts the assumption that $\kappa(G  \setminus w) = 1$. Hence $G \setminus w \in \mathcal{D} \setminus \mathcal{D}_1$, and so $\depth (S/((x_w,y_w)+J_{G \setminus w})) = n$. If $\iv(G\setminus w)=1$ and $d(G\setminus w) =  2$, then $d(G\setminus w) + f(G\setminus w) = n = (n-1) + 2 - \kappa(G \setminus w)$. Suppose  $d(G\setminus w)=3$. This implies  that $\iv(G\setminus w)=2$ and hence $d(G\setminus w) + f(G\setminus w) = n $. 
This implies that $G\setminus w$ belongs to the Hibi-Madani class.
Hence, by \cite[Corollary 2.5]{Hibi-Madani}, $\depth(S/((x_w,y_w)+J_{G\setminus w}))=n$. Thus \cite[Proposition 1.2.9]{bh} applied to the short exact sequence \eqref{ohtani-ses} yields $\depth(S/J_G)\geq n$.

Assume that $\kappa\geq 3$. Then by the chosen $w$, we have $\kappa(G\setminus w) = \kappa -1$ and $\iv(G\setminus w) = \kappa$. Now we show that $d(G\setminus w) = 2$. Assume that $d(G\setminus w)=d'\geq 3$. Let $z,z'\in V(G\setminus w)$ such that $\dist_{G\setminus w}(z,z')=d'$. By \cite[Chapter III, Corollary 6]{Bollobas}, there exist $\kappa -1$ vertex-disjoint paths of length $\geq d'$ from $z$ to $z'$ in $G\setminus w$. Each such path gives at least $(d'-1)$ internal vertices. This implies that $\iv(G \setminus w) \geq (d'-1)(\kappa-1) > \kappa$, a contradiction. Therefore, $d(G\setminus w) = 2$.  Since $\kappa(G\setminus w) = \kappa-1$ by induction, $\depth(S/((x_w,y_w)+J_{G\setminus w}))\geq (n-1)+2-\kappa(G\setminus w)=n+2-\kappa$. Hence, we obtain the required assertion by applying \cite[Proposition 1.2.9]{bh} to \eqref{ohtani-ses}.
\end{proof}

Now we move on to the remaining part, i.e., $\kappa(G)=2$ and $d(G)=3$. Then $\iv(G)=4$. Since $\kappa(G)=2$, there exist two vertex disjoint paths joining $u$ and $v$. Let $P: u, u_1, u_2, v$ and $P':u,  v_1, v_2, v$ be disjoint paths in $G$ joining $u$ and $v$. Note that $u_1, u_2, v_1, v_2$ are the internal vertices of $G$. Since $u$ is a simplicial vertex, there exists a unique clique containing the edges $\{u,u_1\}$  and $\{u,v_1\}$. Similarly there exists a unique clique containing $\{u_2,v\}$ and $\{v_2,v\}$. Let $e_1 = \{u_1, v_1\}, e_2 = \{v_1, v_2\}, e_3=\{u_1, u_2\}$ and $e_4 = \{v_2, u_2\}$. 
Consider the following graph $H''$:

\captionsetup[figure]{labelformat=empty}
\begin{figure}[H]
				\begin{tikzpicture}[scale=1]
				
				\draw (0,1.02)-- (1.36,1.02);
\draw (1.36,1.02)-- (1.36,0);
\draw (0,1.02)-- (0,0);
\draw (0,0)-- (1.36,0);
\draw (0,0)-- (1.36,1.02);
\draw (0,1.02)-- (1.36,0);
\draw (1.36,0)-- (2.26,0.5);
\draw (1.36,1.02)-- (2.26,0.5);
\draw (-0.9,0.5)-- (0,0);
\draw (-0.9,0.5)-- (0,1.02);
\begin{scriptsize}
\fill  (0,0) circle (1.5pt);
\draw (0.16,-0.2) node {$v_1$};
\fill  (1.36,0) circle (1.5pt);
\draw (1.52,-0.2) node {$v_2$};
\fill  (2.26,0.5) circle (1.5pt);
\draw (2.36,0.2) node {$v$};
\fill (0,1.02) circle (1.5pt);
\draw (-0.1,1.28) node {$u_1$};
\fill  (1.36,1.02) circle (1.5pt);
\draw (1.5,1.28) node {$u_2$};
\fill  (-0.9,0.5) circle (1.5pt);
\draw (-0.8,0.2) node {$u$};
\draw (-0.2,0.5) node {$e_1$};
\draw (0.7,-0.2) node {$e_2$};
\draw (0.7,1.28) node {$e_3$};
\draw (1.56,0.5) node {$e_4$};
\end{scriptsize}
			
			\end{tikzpicture}
\caption{$H''$}
\end{figure}

Now we assume that $G$ is chordal. Then the cycle on the vertices $\{u_1,u_2,v_2,v_1\}$ is not an induced cycle, i.e., $\{u_1,v_2\} \in E(G)$ or $\{u_2,v_1\} \in E(G)$. Then $G$ is a clique sum of complete graphs along cliques of size at least $2$ of the induced subgraph $G[u_1,u_2,v_1,v_2]$. Moreover, there are unique cliques containing $\{u,u_1,v_1\}$ and $\{v,u_2,v_2\}$.


\begin{minipage}{\linewidth}
		\begin{minipage}{.5\linewidth}
		
		\captionsetup[figure]{labelformat=empty}
\begin{figure}[H]
				\begin{tikzpicture}[scale=1]
				
				\draw (0,1.02)-- (1.36,1.02);
\draw (1.36,1.02)-- (1.36,0);
\draw (0,1.02)-- (0,0);
\draw (0,0)-- (1.36,0);
\draw (0,0)-- (1.36,1.02);
\draw (0,1.02)-- (1.36,0);
\draw (1.36,0)-- (2.58,0);
\draw (1.36,0)-- (2.58,0);
\draw (1.36,0)-- (2.58,0);
\draw (1.36,1.02)-- (2.58,0);
\draw (1.36,0)-- (2.58,1.02);
\draw (1.36,1.02)-- (2.58,1.02);
\draw (0.68,1.96)-- (0,1.02);
\draw (0.68,1.96)-- (1.36,1.02);
\draw (0.68,1.96)-- (0,0);
\draw (-0.9,0.5)-- (0,0);
\draw (-0.9,0.5)-- (0,1.02);
\begin{scriptsize}
\fill  (0,0) circle (1.5pt);
\draw (0.16,-0.2) node {$v_1$};
\fill  (1.36,0) circle (1.5pt);
\draw (1.52,-0.2) node {$v_2$};
\fill  (2.58,0) circle (1.5pt);
\draw (2.74,0.26) node {$v$};
\fill (0,1.02) circle (1.5pt);
\draw (-0.1,1.28) node {$u_1$};
\fill  (1.36,1.02) circle (1.5pt);
\draw (1.5,1.28) node {$u_2$};
\fill  (2.58,1.02) circle (1.5pt);
\fill  (0.68,1.96) circle (1.5pt);
\fill  (-0.9,0.5) circle (1.5pt);
\draw (-0.8,0.2) node {$u$};
\draw (-0.2,0.5) node {$e_1$};
\draw (0.7,-0.2) node {$e_2$};
\draw (0.7,1.28) node {$e_3$};
\draw (1.56,0.5) node {$e_4$};
\end{scriptsize}
			
			\end{tikzpicture}
\caption{$G_9$}
\end{figure}
		\end{minipage}
		\begin{minipage}{.5\linewidth}
		\captionsetup[figure]{labelformat=empty}
\begin{figure}[H]
				\begin{tikzpicture}[scale=1]	
				
				\draw (0,1.02)-- (1.36,1.02);
\draw (1.36,1.02)-- (1.36,0);
\draw (0,1.02)-- (0,0);
\draw (0,0)-- (1.36,0);
\draw (0,0)-- (1.36,1.02);
\draw (0,1.02)-- (1.36,0);
\draw (1.36,0)-- (2.58,0);
\draw (1.36,0)-- (2.58,0);
\draw (1.36,0)-- (2.58,0);
\draw (1.36,1.02)-- (2.58,0);
\draw (1.36,0)-- (2.58,1.02);
\draw (1.36,1.02)-- (2.58,1.02);
\draw (0.68,1.96)-- (0,1.02);
\draw (0.68,1.96)-- (1.36,1.02);
\draw (0.68,1.96)-- (0,0);
\draw (-0.9,0.5)-- (0,0);
\draw (-0.9,0.5)-- (0,1.02);
\draw [color=red](0.68,1.96)-- (-0.9,0.5);
\draw [color=red](-0.9,0.5)-- (1.36,1.02);
\draw [color=red](-0.9,0.5)-- (1.36,0);
\draw [color=red](0.68,1.96)-- (1.36,0);
\begin{scriptsize}
\fill  (0,0) circle (1.5pt);
\draw (0.16,-0.2) node {$v_1$};
\fill  (1.36,0) circle (1.5pt);
\draw (1.52,-0.2) node {$v_2$};
\fill  (2.58,0) circle (1.5pt);
\draw (2.74,0.26) node {$v$};
\fill  (0,1.02) circle (1.5pt);
\draw (-0.35,1.28) node {$u_1$};
\fill  (1.36,1.02) circle (1.5pt);
\draw (1.5,1.28) node {$u_2$};
\fill  (2.58,1.02) circle (1.5pt);
\fill  (0.68,1.96) circle (1.5pt);
\fill  (-0.9,0.5) circle (1.5pt);
\draw (-0.8,0.2) node {$u$};
\end{scriptsize}
				
				\end{tikzpicture}
				\caption{$(G_9)_{u_1}$}
\end{figure}
		\end{minipage}
	\end{minipage}

\begin{theorem}\label{claim6}
Let $G$ be a chordal graph $\kappa(G)=2$, $d(G)=3$ and $H''$ is an induced subgraph of $G$. If there are no cliques in $G$ along the edges $e_2$ and $e_3$ of $H''$, then $\depth(S/J_G)=n-1.$
\end{theorem}
\begin{proof}
First observe that $d(G)+f(G) = n-1 \leq \depth S/J_G$. For the reverse inequality,
we prove that $\beta_{n+1,n+4}(S/J_G)\neq 0$. Let $w=u_1$. Then, $G_w$ is a clique sum of complete graphs along the edge $e_4$. Therefore, $G_w$ and $G_w\setminus w$ belong to the Hibi-Madani class $\mathcal{F}_2$. Therefore, by  \Cref{extremal-hibi-madani} and \eqref{Bettiproduct}, $\beta_{n,n+2}(S/J_{G_w})$ and $\beta_{n+1,n+3}(S/((x_w,y_w)+J_{G_w\setminus w}))$ are extremal Betti numbers. 
We now describe the graph $G \setminus w$. The unique clique containing $\{u,v_1\}$ is attached to the rest of the graph along the vertex $v_1$. Since $\kappa(G) = 2$ and there are no cliques attached to $e_2$ and $e_3$ in $G$, it can be seen that there are no cliques attached to $v_2$ and $u_2$ in $G\setminus w$. Hence, in $G\setminus w$, complete graphs are attached along the complete subgraphs of $G[u_2,v_1,v_2]$, except along the vertices $u_2,v_2$ and the edge $e_2$. Moreover, $u_2,v_1,v_2$ are the only internal vertices in $G \setminus w$. Hence, $d(G\setminus w) = 3.$ Since $N_{G\setminus w}(u_2) = \{v_1,v_2,v\}$ and there are no cliques along the vertices $u_2$ and $v_2$,  $G\setminus w\in \mathcal{D}_1$. Therefore, by the proof of \Cref{claim2} and \eqref{Bettiproduct}, $\beta_{n+1,n+4}(S/((x_w,y_w)+J_{G\setminus w}))\neq 0$. 
Now it follows from \eqref{ohtani-tor} that $\beta_{n+1,n+4}(S/J_G)\neq 0$.
\end{proof}

\begin{minipage}{\linewidth}
		\begin{minipage}{.5\linewidth}
		
		\captionsetup[figure]{labelformat=empty}
\begin{figure}[H]
				\begin{tikzpicture}[scale=1]
				
				\draw (0,1.02)-- (1.36,1.02);
\draw (1.36,1.02)-- (1.36,0);
\draw (0,1.02)-- (0,0);
\draw (0,0)-- (1.36,0);
\draw (0,1.02)-- (1.36,0);
\draw (0,1.02)-- (-0.9,0.5);
\draw (-0.9,0.5)-- (0,0);
\draw (0,0)-- (0.64,-0.66);
\draw (0.64,-0.66)-- (1.36,0);
\draw (1.36,1.02)-- (2.22,0.62);
\draw (2.22,0.62)-- (1.36,0);
\begin{scriptsize}
\fill  (0,0) circle (1.5pt);
\draw (0.16,0.26) node {$v_1$};
\fill  (1.36,0) circle (1.5pt);
\draw (1.52,0.26) node {$v_2$};
\fill (0,1.02) circle (1.5pt);
\draw (0.16,1.28) node {$u_1$};
\fill  (1.36,1.02) circle (1.5pt);
\draw (1.5,1.28) node {$u_2$};
\fill (-0.9,0.5) circle (1.5pt);
\draw (-0.8,0.2) node {$u$};
\fill (0.64,-0.66) circle (1.5pt);
\fill  (2.22,0.62) circle (1.5pt);
\draw (2.38,0.88) node {$v$};
\end{scriptsize}
			
			\end{tikzpicture}
\caption{$G_{10}$}
\end{figure}
		\end{minipage}
		\begin{minipage}{.5\linewidth}
		\captionsetup[figure]{labelformat=empty}
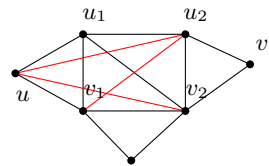
\begin{figure}[H]
				\begin{tikzpicture}[scale=1]	
				
				\draw (0,1.02)-- (1.36,1.02);
\draw (1.36,1.02)-- (1.36,0);
\draw (0,1.02)-- (0,0);
\draw (0,0)-- (1.36,0);
\draw (0,1.02)-- (1.36,0);
\draw (0,1.02)-- (-0.9,0.5);
\draw (-0.9,0.5)-- (0,0);
\draw (0,0)-- (0.64,-0.66);
\draw (0.64,-0.66)-- (1.36,0);
\draw (1.36,1.02)-- (2.22,0.62);
\draw (2.22,0.62)-- (1.36,0);
\draw [color=red](-0.9,0.5)-- (1.36,0);
\draw [color=red](-0.9,0.5)-- (1.36,1.02);
\draw [color=red](1.36,1.02)-- (0,0);
\begin{scriptsize}
\fill (0,0) circle (1.5pt);
\draw (0.16,0.26) node {$v_1$};
\fill  (1.36,0) circle (1.5pt);
\draw (1.52,0.26) node {$v_2$};
\fill  (0,1.02) circle (1.5pt);
\draw (0.16,1.28) node {$u_1$};
\fill  (1.36,1.02) circle (1.5pt);
\draw (1.5,1.28) node {$u_2$};
\fill  (-0.9,0.5) circle (1.5pt);
\draw (-0.8,0.2) node {$u$};
\fill  (0.64,-0.66) circle (1.5pt);
\fill  (2.22,0.62) circle (1.5pt);
\draw (2.38,0.88) node {$v$};
\end{scriptsize}
				
				\end{tikzpicture}
				\caption{$(G_{10})_{u_1}$}
\end{figure}
		\end{minipage}
	\end{minipage}

\vskip 2mm \noindent
\begin{theorem}\label{claim7}
Let $G$ be a chordal graph with $\kappa(G)=2$ and $d(G)=3$. Assume that either of the following statements hold:
\begin{enumerate}
    \item $G[u_1,u_2,v_1,v_2]$ is not a complete graph;
    \item $H''$ is an induced subgraph of $G$ and there are cliques in $G$ along the edges $e_2$ or $e_3$.
\end{enumerate}
Then $\depth(S/J_G)=n.$
\end{theorem}
\begin{proof}
We show that $\depth(S/J_G)\geq n$. 
First assume that $G[u_1,u_2,v_1,v_2]$ is not a complete graph. Assume that $\{u_1,v_2\} \in E(G)$ and $\{u_2,v_1\} \notin E(G)$. Take $w=u_1$. If there are cliques along $e_2$, then $G_w$ is a graph obtained by taking clique sum of the complete graph on $N_{G}[w]$ and complete graphs along $e_2$ and $e_4$. If $d(G_w) \geq 3$, then there exist two disjoint paths of length at least $d(G_w)$ connecting any two vertices $z$ and $z'$ such that $\dist_{G_w}(z, z') = d(G_w)$. This gives rise to at least $4$ internal vertices. This contradicts the fact that $\iv(G_w)\leq 3$. Therefore, $\kappa(G_w) = 2 = d(G_w)$. Since $f(G_w)\geq f(G)+1$, we can see that $n \leq d(G_w)+f(G_w) + 1 \leq n+2-\kappa(G_w)=n$. It may also be noted that $G_w \setminus w$ also satisfy the same property as $G_w$. Hence by \Cref{claim5}, $\depth (S/J_{G_w}) = n$ and $\depth (S/((x_w,y_w)+J_{G_w \setminus w}))=n-1$. If there are no cliques along $e_2$, then $G_w$ is a clique sum of the complete graph on $N_G[w]$ and complete graphs along $e_4$. Then $G_w$ and $G_w \setminus w$ belong to the Hibi-Madani class $\mathcal{F}_2$. Hence by \cite[Corollary 2.5]{Hibi-Madani}, we get $\depth (S/J_{G_w}) = n$ and $\depth (S/((x_w,y_w)+J_{G_w \setminus w})) = n-1.$
Now we study the graph $G \setminus w$. Suppose, in $G$, there are cliques along  $e_3$ or at least one clique other than the clique containing $\{u_2,v_2,v\}$ along $e_4$ or cliques along $G[u_1,u_2,v_2]$. Then $u_2$ remains an internal vertex in $G \setminus w$. By taking the path $P': u,v_1,v_2,v$, we can see that $N_{P'}(u_2) = \{v_2,v\}$. Therefore, $G\setminus w \in \mathcal{D}\setminus \mathcal{D}_1$, so that $\depth (S/((x_w,y_w)+J_{G\setminus w})) = n$, by \Cref{claim1}. If there are no cliques along $e_3$, no cliques other than the clique containing $\{u_2,v_2,v\}$ along $e_4$ and no cliques along $G[u_1,u_2,v_2]$, then $u_2$ becomes a simplicial vertex in $G\setminus w$. Therefore, $G\setminus w$ belongs to the Hibi-Madani class $\mathcal{G}_3$. Hence by \cite[Corollary 2.5]{Hibi-Madani}, $\depth (S/((x_w,y_w)+J_{G\setminus w})) = n$. Therefore, if $G[u_1,u_2,v_1,v_2]$ is not a complete graph, then it follows from \cite[Propsotion 1.2.9]{bh} and \eqref{ohtani-ses} that $\depth (S/J_G) \geq n$.

Now we assume that $G[u_1,u_2,v_1,v_2]$ is a complete graph and cliques are attached along $e_2$ or $e_3$. If there are cliques along the edge $e_3$, then take $w=u_1$, else take $w=v_1$. If there are no cliques along $e_2$, then $G_w$ is a clique sum of complete graphs along $e_4$. Hence $G_w$ and $G_w\setminus w$ belong to the Hibi-Madani class $\mathcal{F}_2$. If there are cliques along $e_2$, then $G_w$ is obtained by taking clique sum of the complete graph on $N_G[w]$ and complete graphs along $e_2$ and $e_4$. Hence $G_w$ and $G_w\setminus w$ have the property that $d(G_w) = d(G_w \setminus w) = \kappa(G_w) = \kappa(G_w \setminus w) = 2.$ Moreover, as seen in the earlier paragraph, we can see that $d(G_w) + f(G_w) +1 = n$. Therefore, by either \cite[Corollary 2.5]{Hibi-Madani} or \Cref{claim5}, we get $\depth (S/J_{G_w}) = n$ and $\depth (S/((x_w,y_w)+J_{G_w\setminus w}))=n-1.$ If $w = u_1,$ then $G\setminus w \in \mathcal{D}\setminus \mathcal{D}_1$ with the path $P':u,v_1,v_2,v$ and $u_2$ being an internal vertex. If $w = v_1$, then again $G \setminus w \in \mathcal{D}\setminus \mathcal{D}_1$ with the path $P'':u,u_1,u_2,v$ and $v_2$ being an internal vertex. Hence by \Cref{claim1}, $\depth (S/((x_w,y_w)+J_{G\setminus w})) = n$.
Hence by applying \cite[Proposition 1.2.9]{bh} to the short exact sequence \eqref{ohtani-ses}, we get $\depth (S/J_G) \geq n$.
\end{proof}

Now assume that $G$ is not chordal. Then $G$ contains precisely one induced cycle of length $4$. Since $\iv(G) = 4$, the only possibility of a $C_4$ in $G$ is on the vertices $\{u_1, u_2, v_1, v_2\}$. Let $C$ denote the cycle on $u_1,u_2,v_1,v_2$. Then $G$ is a repeated clique sum of complete graphs along the edges of the cycle $C$.

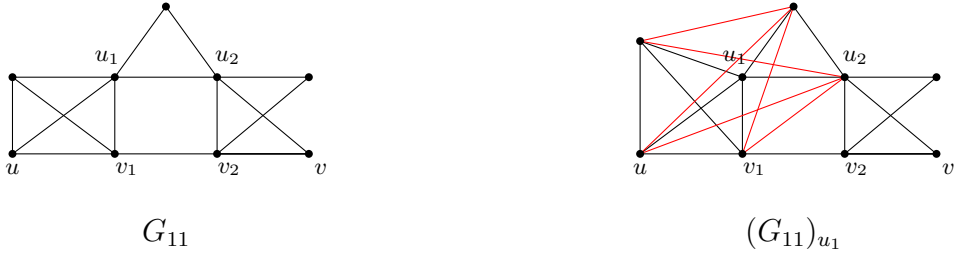
\begin{minipage}{\linewidth}
		\begin{minipage}{.5\linewidth}
		
		\captionsetup[figure]{labelformat=empty}
\begin{figure}[H]
				\begin{tikzpicture}[scale=1]
				
				\draw (0,1.02)-- (1.36,1.02);
\draw (1.36,1.02)-- (1.36,0);
\draw (0,1.02)-- (0,0);
\draw (0,0)-- (1.36,0);
\draw (1.36,0)-- (2.58,0);
\draw (1.36,0)-- (2.58,0);
\draw (1.36,0)-- (2.58,0);
\draw (1.36,1.02)-- (2.58,0);
\draw (1.36,0)-- (2.58,1.02);
\draw (1.36,1.02)-- (2.58,1.02);
\draw (0.68,1.96)-- (0,1.02);
\draw (0.68,1.96)-- (1.36,1.02);
\draw (-1.36,0)-- (0,0);
\draw (-1.36,0)-- (0,1.02);
\draw (-1.36,1.02)-- (0,0);
\draw (-1.36,1.02)-- (0,1.02);
\draw (-1.36,1.02)-- (-1.36,0);
\begin{scriptsize}
\fill  (0,0) circle (1.5pt);
\draw (0.16,-0.2) node {$v_1$};
\fill  (1.36,0) circle (1.5pt);
\draw (1.52,-0.2) node {$v_2$};
\fill  (2.58,0) circle (1.5pt);
\draw (2.74,-0.2) node {$v$};
\fill (0,1.02) circle (1.5pt);
\draw (-0.1,1.28) node {$u_1$};
\fill  (1.36,1.02) circle (1.5pt);
\draw (1.5,1.28) node {$u_2$};
\fill  (2.58,1.02) circle (1.5pt);
\fill  (0.68,1.96) circle (1.5pt);
\fill  (-1.36,0) circle (1.5pt);
\draw (-1.36,-0.2) node {$u$};
\fill  (-1.36,1.02) circle (1.5pt);

\end{scriptsize}
			
			\end{tikzpicture}
\caption{$G_{11}$}
\end{figure}
		\end{minipage}
		\begin{minipage}{.5\linewidth}
		\captionsetup[figure]{labelformat=empty}
\begin{figure}[H]
				\begin{tikzpicture}[scale=1]	
				
				\draw (0,1.02)-- (1.36,1.02);
\draw (1.36,1.02)-- (1.36,0);
\draw (0,1.02)-- (0,0);
\draw (0,0)-- (1.36,0);
\draw (1.36,0)-- (2.58,0);
\draw (1.36,0)-- (2.58,0);
\draw (1.36,0)-- (2.58,0);
\draw (1.36,1.02)-- (2.58,0);
\draw (1.36,0)-- (2.58,1.02);
\draw (1.36,1.02)-- (2.58,1.02);
\draw (0.68,1.96)-- (0,1.02);
\draw (0.68,1.96)-- (1.36,1.02);
\draw (-1.36,0)-- (0,0);
\draw (-1.36,0)-- (0,1.02);
\draw (-1.36,1.5)-- (0,0);
\draw (-1.36,1.5)-- (0,1.02);
\draw (-1.36,1.5)-- (-1.36,0);
\draw [color=red] (-1.36,1.5)-- (0.68,1.96);
\draw [color=red] (-1.36,1.5)-- (1.36,1.02);
\draw [color=red] (-1.36,0)-- (0.68,1.96);
\draw [color=red] (0,0)-- (1.36,1.02);
\draw [color=red] (0,0)-- (0.68,1.96);
\draw [color=red] (-1.36,0)-- (1.36,1.02);
\begin{scriptsize}
\fill  (0,0) circle (1.5pt);
\draw (0.16,-0.2) node {$v_1$};
\fill  (1.36,0) circle (1.5pt);
\draw (1.52,-0.2) node {$v_2$};
\fill  (2.58,0) circle (1.5pt);
\draw (2.74,-0.2) node {$v$};
\fill (0,1.02) circle (1.5pt);
\draw (-0.1,1.28) node {$u_1$};
\fill  (1.36,1.02) circle (1.5pt);
\draw (1.5,1.28) node {$u_2$};
\fill  (2.58,1.02) circle (1.5pt);
\fill  (0.68,1.96) circle (1.5pt);
\fill  (-1.36,0) circle (1.5pt);
\draw (-1.36,-0.2) node {$u$};
\fill  (-1.36,1.5) circle (1.5pt);

\end{scriptsize}
				
				\end{tikzpicture}
				\caption{$(G_{11})_{u_1}$}
\end{figure}
		\end{minipage}
	\end{minipage}

\begin{theorem}\label{claim8}
Let $G$ be a non-chordal graph with $\kappa(G)=2$ and $d(G)=3$. Then
$\depth(S/J_G)=n.$
\end{theorem}
\begin{proof}
Let $w=u_1$. We first describe the graph $G_w$. Since $G$ is the clique sum of complete graphs along the edges of the cycle $C$, $G_w$ is obtained by taking clique sum of complete graphs along the edges of the $K_3$ on the vertices $\{v_1,u_2,v_2\}$. Therefore, these are the only internal vertices of $G_w$ and hence $\iv(G_w) = 3$. If $d(G_w) \geq 3$, then there exist two vertex disjoint paths of length at least $d(G_w)$ which gives rise to at least $4$ internal vertices, a contradiction. Hence $d(G_w) = 2$. Since $G_w \setminus \{u_2,v_2\}$ is disconnected, $\kappa(G_w) = 2$. And, $d(G_w)+f(G_w)+1 = n$. Moreover, $G_w\setminus w$ also satisfies the same properties. Hence \Cref{claim5}, $\depth(S/J_{G_w})=n$ and $\depth(S/((x_w,y_w)+J_{G_w\setminus w}))=n-1$.

If there is a clique along $e_3$ or at least two cliques along $e_4$, then $G \setminus w \in \mathcal{D}\setminus \mathcal{D}_1$ with the path $P':u,v_1,v_2,v$ and $u_2$ being an internal vertex outside the path $P$. If there are no cliques along $e_3$ and there is a unique clique, namely the clique containing $v$, along the edge $e_4$, then $G\setminus w$ is a repeated clique sum of complete graphs along the vertices and edges of the path $P':u,v_1,v_2,v$. Hence, $G\setminus w$ belongs to the Hibi-Madani class $\mathcal{G}_3$. Therefore, either by using \Cref{claim1} or by using \cite[Corollary 2.5]{Hibi-Madani}, we conclude that $\depth(S/((x_w,y_w)+J_{G\setminus w}))=n$. Now the assertion follows from \cite[Proposition 1.2.9]{bh} applied to the short exact sequence \eqref{ohtani-ses}.
\end{proof}

%
\section{On the lower bound of depths}
It was proved by Rouzbahani Malayeri, Saeedi Madni and Kiani  that for any graph $G$, $d(G)+f(G) \leq \depth(S/J_G)$, \cite{MKM-depth}.
There is a natural question arising from this result, namely, ``\textit{What are the graphs $G$ satisfying the equality $\depth(S/J_G) = d(G)+f(G)$}?''. While it is not clear, whether one can give a complete characterization of graphs with this minimal depth, in this section, we try to understand when a graph does not have this minimal depth.

A \textit{block} of a graph is a non-trivial maximal connected subgraph which has no cut vertex. A graph $G$ is said to be a \textit{block graph} if every block of $G$ is a complete subgraph. It is easy to observe that $G$ is a block graph if and only if G is a chordal graph
with the property that any two maximal cliques intersect in at most one vertex. A
connected chordal graph $G$ is called a \textit{generalized block graph} if three maximal
cliques of $G$ intersect non-trivially, then the intersection of each pair of them is the
same.

Let $G$ be a generalized block graph on $[n]$. The
maximum size of the maximal cliques of $G$ is called the clique number of a graph $G$, denoted by $\omega(G)$. For each $i=1,\dots,\omega(G)-1$, we set
\[ \mathcal{A}_i=\{ A\subseteq [n]: |A|=i \text{ and } A \text{ is a minimal cut set of } G\}.
\]
Also, we set $a_i=|\mathcal{A}_i|.$ Note that a generalized block graph $G$ is a block graph if and only if $a_i=0$ for $i\geq 2$. Let $m(G)$ denote the number of minimal cut sets of $G$, i.e., $m(G)=\sum_{i=1}^{\omega(G)-1}a_i(G).$
\begin{remark}\label{depth-lower}
Let $G$ be a generalized block graph. Then $\depth(S/J_G)=d(G)+f(G)$ if and only if $m(G)+1=d(G).$
\end{remark}
\begin{proof}
It follows from \cite[Theorem 3.2]{KM-CA} that 
\begin{equation*}
    \depth(S/J_G)=n+1-\sum_{i=2}^{\omega(G)-1}(i-1)a_i(G) =n-\sum_{i=1}^{\omega(G)-1}ia_i(G)+1+m(G).
\end{equation*}
Note that $\iv(G)=\sum_{i=1}^{\omega(G)-1}ia_i(G)$. Hence, the assertion follows.
\end{proof}
Let $G$ be a connected chordal graph and $\Delta(G)$ be its clique complex. Then $\Delta(G)$ is a quasi-tree i.e., there exists a leaf order of the facets in $\Delta(G)$. Let $G$ be such that $F_1,\dots,F_r$ be a leaf order in $\Delta(G)$, and $F_{i-1}$ is the unique branch of $F_{i}$ for $i=2,\dots,r$. Ene, Herzog and Hibi studied this class of graphs in \cite{her1} and called it a \textit{chain of cliques}. Then we can write 
\[ G=K_{r_1}\cup_{K_{q_1}}K_{r_2}\cup_{K_{q_2}}\cdots \cup_{K_{q_{m-1}}}K_{r_m}, m\geq 2.
\]
We now characterize $G$ such that $\depth(S/J_G)=d(G)+f(G)$. If $F_{i-2}\cap F_i=\emptyset$ for all $i=3,\dots,r$, i.e., $V(K_{q_{j-1}})\cap V(K_{q_{j}})=\emptyset$ for all $j=2,\dots,r$, then $G$ is a generalized block graph with $m(G)+1=d(G)$, and hence by \Cref{depth-lower}, $\depth(S/J_G)=d(G)+f(G)$.

We now prove that if $V(K_{q_{j-1}})\cap V(K_{q_{j}})\neq \emptyset$ for some $j=2,\dots,r$, then $d(G)+f(G)+1\leq \depth(S/J_G)$. First we observe that we need to prove the result only for indecomposable graphs, i.e., $q_i > 1$ for all $i = 1, \ldots, m-1$. Assume that the result is proved for indecomposable graphs. Let $G = G_1 \cup_{w} G_2$, where $G, G_1$ and $G_2$ are chain of cliques, $G_1, G_2$ indecomposable, and $w$ is a simplicial vertex in $G_1$ and in $G_2$. Write $G = G_1 \cup_{K_{q_i}} G_2$, where $G_1 = K_{r_1}\cup_{K_{q_1}} \cdots \cup_{K_{q_{i-1}}}K_{r_i}$ and $G_2 = K_{r_{i+1}}\cup_{K_{q_{i+1}}} \cdots \cup_{K_{q_{m-1}}}K_{r_m}.$ Then by \cite[Theorem 2.7]{Rinaldo-Rauf}, $\depth (S/J_G)=\depth (S_{G_1}/J_{G_1}) + \depth (S_{G_2}/J_{G_2})-2$, where $S_{G_i}$ denotes the polynomial ring in $2|V(G_i)|$ variables. If $G$ satisfies the property that $V(K_{q_{j-1}}) \cap V(K_{q_j}) \neq \emptyset$, then it is clear that $G_1$ or $G_2$ satisfies this property. Without loss of generality, we may assume that $G_1$ satisfies this property. 
If $K_{q_i} = \{w\}$, then $w$ is an internal vertex in $G$ while it is a simplicial vertex in $G_1$ and in $G_2$. Moreover, if $v \in V(G_i) \setminus \{w\}$, then $v$ is simplicial in $G_i$ if and only if $v$ is simplicial in $G$. Therefore, $f(G) = |V(G)|-\iv(G) = (|V(G_1)|+|V(G_2)|-1) - (\iv(G_1)+\iv(G_2)+1) = f(G_1)+f(G_2)-2$.  Since any longest induced path in $G$ must pass through $w$, $d(G) = d(G_1) + d(G_2)$. Therefore,
\begin{eqnarray*}
\depth (S/J_G) & = & \depth (S_{G_1}/J_{G_1}) + \depth (S_{G_2}/J_{G_2}) -2 \\
& \geq & [d(G_1) + f(G_1) + 1] + [d(G_2) +  f(G_2) ] -2 \\
& = & d(G) + f(G) + 1.
\end{eqnarray*}
Therefore, to prove our result, we only need to prove the same for indecomposable graphs. Hence, we assume that $q_i \geq 2$ for all $i =  1, \ldots, m-1$.

\begin{minipage}{\linewidth}
	\begin{minipage}{.5\linewidth}
		
		\captionsetup[figure]{labelformat=empty}
\begin{figure}[H]
				\begin{tikzpicture}[scale=1]
\draw (0,1)-- (1.3,1);
\draw (1.3,0)-- (0,1);
\draw (0,1)-- (0,0);
\draw (0,0)-- (1.3,0);
\draw (1.3,0)-- (1.3,1);
\draw (0,0)-- (1.3,1);
\draw (2.6,1)-- (1.3,1);
\draw (1.3,0)-- (2.6,0);
\draw (2.6,1)-- (2.6,0);
\draw (1.3,0)-- (2.6,1);
\draw (1.3,1)-- (2.6,0);
\draw (2.6,0)-- (3.35,-0.3);
\draw (3.38,0.6)-- (2.6,1);
\draw (3.38,0.6)-- (2.6,0);
\draw (3.38,0.6)-- (3.35,-0.3);
\begin{scriptsize}
\fill  (0,0) circle (1.5pt);
\fill (1.3,0) circle (1.5pt);
\fill  (2.6,0) circle (1.5pt);
\draw (2.76,-0.3) node {$w$};
\fill  (3.35,-0.3) circle (1.5pt);
\fill  (0,1) circle (1.5pt);
\fill  (1.3,1) circle (1.5pt);
\fill  (2.6,1) circle (1.5pt);
\fill  (3.38,0.6) circle (1.5pt);
\end{scriptsize}
\end{tikzpicture}
\caption{$G_{12}$}
\end{figure}
		\end{minipage}
		\begin{minipage}{.5\linewidth}
		\captionsetup[figure]{labelformat=empty}
\begin{figure}[H]
				\begin{tikzpicture}[scale=1]
\draw (0,1)-- (1.3,1);
\draw (1.3,0)-- (0,1);
\draw (0,1)-- (0,0);
\draw (0,0)-- (1.3,0);
\draw (1.3,0)-- (1.3,1);
\draw (0,0)-- (1.3,1);
\draw (2.6,1)-- (1.3,1);
\draw (1.3,0)-- (2.6,0);
\draw (2.6,1)-- (2.6,0);
\draw (1.3,0)-- (2.6,1);
\draw (1.3,1)-- (2.6,0);
\draw (2.6,0)-- (3.35,-0.3);
\draw (3.38,0.6)-- (2.6,1);
\draw (3.38,0.6)-- (2.6,0);
\draw (3.38,0.6)-- (3.35,-0.3);
\draw [color=red](3.38,0.6)-- (1.3,1);
\draw [color=red](3.38,0.6)-- (1.3,0);
\draw [color=red](3.35,-0.3)-- (1.3,0);
\draw [color=red](3.35,-0.3)-- (2.6,1);
\draw [color=red](3.35,-0.3)-- (1.3,1);
\begin{scriptsize}
\fill  (0,0) circle (1.5pt);
\fill  (1.3,0) circle (1.5pt);
\fill (2.6,0) circle (1.5pt);
\draw (2.76,-0.4) node {$w$};
\fill  (3.35,-0.3) circle (1.5pt);
\fill  (0,1) circle (1.5pt);
\fill  (1.3,1) circle (1.5pt);
\fill  (2.6,1) circle (1.5pt);
\fill  (3.38,0.6) circle (1.5pt);
\end{scriptsize}
\end{tikzpicture}
\caption{$(G_{12})_w$}
\end{figure}
		\end{minipage}
	\end{minipage}

\begin{theorem}\label{coc}
Let $G=K_{r_1}\cup_{K_{q_1}}K_{r_2}\cup_{K_{q_2}}\cdots \cup_{K_{q_{m-1}}}K_{r_m}$ such that $q_i\geq 2$ for all $i=1,\dots,m-1$ and $V(K_{q_{j-1}})\cap V(K_{q_{j}})\neq \emptyset$ for some $j=2,\dots,r$. Then $\depth(S/J_G)\geq d(G)+f(G)+1$.
\end{theorem}
\begin{proof} First, observe that $\iv(G)=|\cup_{i=1}^{m-1} V(K_{q_i})|$ and $\kappa(G)=\min\{q_i : i=1,\ldots,m-1\}$.
We prove the assertion by induction on $\iv(G)$. The base case is $\iv(G)=3$, i.e., $m=3, q_1=2$ and $q_2=2$. By construction, $d(G) = 2, f(G) = n-3$ and $\kappa(G)=2$. Hence $d(G)+f(G)+1 = n = n+2-\kappa(G)$. Therefore, by \Cref{claim5}, $\depth (S/J_G) = n=d(G)+f(G)+1$. Assume that $\iv(G)\geq 4$. Let $v\in V(K_{q_{j-1}})\cap V(K_{q_{j}})$. 
Then \[G_v = K_{r_1} \cup_{K_{q_1}} \cdots \cup_{K_{q_{j-3}}} K_{r_{j-2}} \cup_{K_{q_{j-2}}} K_r \cup_{K_{q_{j+1}}}K_{r_{j+2}} \cup_{K_{q_{j+2}}} \cdots \cup_{K_{q_{m-1}}} K_{r_m},\]
where $K_r$ is the complete graph on $N_G[v]$.
If $V(K_{q_{k-1}})\cap V(K_{q_{k}}) = \emptyset$ for all $k\neq j$, then $G_v$ and $G_v \setminus v$ are generalized block graphs. 
If $V(K_{q_{k-1}})\cap V(K_{q_{k}}) \neq \emptyset$ for some $k \neq j$, then 
$G_v$ and $G_v\setminus v$ are graphs satisfying the hypothesis of the theorem with $\iv(G_v) < \iv(G)$ and $\iv(G_v\setminus v) < \iv(G)$.
Any longest induced path in $G$ intersects with $K_{r_t}$ for all $t = 1,\ldots,m$ and can have at most one edge in $K_{r_t}$ for each $t=1,\ldots,m$. In any of the longest induced path in $G$, there is a portion of length at least $2$ that passes through $K_{r_{j-1}} \cup K_{r_j} \cup K_{r_{j+1}}$. Since these three complete graphs have merged to become one complete graph in $G_v$, such a portion will be replaced with an edge in $G_v$. All the other edges in the path in $G$ will remain the same in $G_v$ as well. Hence,
$d(G_v) \leq d(G)-1$. Combining this inequality with \Cref{diameter-lemma}, we get $d(G_v) = d(G)-1$. 

Note that the vertices of $V(K_{q_i})$ are internal vertices in $G$ for $i = 1, \ldots, m-1$. If $v  \in V(K_{q_i})$, then the vertices of $V(K_{q_i})$ become simplicial vertices in $G_v$. Hence $f(G_v) =f(G)+|\cup_{v \in V(K_{q_l})}V(K_{q_l})|.$
Since $|V(K_{q_i})| \geq 2$ for all $i$ and $v \in V(K_{q_{j-1}}) \cap V(K_{q_j})$, $|\cup_{v \in V(K_{q_l})}V(K_{q_l})|\geq 3$. Moreover, all these properties are satisfied by the graph $G_v \setminus v$ as well. Therefore,
$d(G_v)+f(G_v)\geq d(G)+f(G)+2$ and $d(G_v\setminus v)+f(G_v\setminus v)\geq d(G)+f(G)+1$. Therefore, by \cite[Theorem 3.5]{MKM-depth}, we have $$\depth(S/J_{G_v})\geq d(G_v)+f(G_v)\geq d(G)+f(G)+2 \text{ and } $$ $$\depth(S/((x_v,y_v)+J_{G_v\setminus v}))\geq d(G_v\setminus v)+f(G_v\setminus v)\geq  d(G)+f(G)+1.$$
Now we describe $G \setminus v$. Suppose $V(K_{q_{j-1}}) \cap V(K_{q_j}) = \{v\}$. Since $G\setminus v$ is connected, $f(G\setminus v)=f(G)$. Note that any shortest induced path between two vertices, say $z$ and $z'$, at a distance $d(G)$ must pass through $v$. Thus $\dist_{G\setminus v}(z,z') \geq d(G)+1$. Therefore, $d(G\setminus v) \geq d(G)+1$. Hence 
\[\depth (S/((x_v,y_v)+J_{G\setminus v})) \geq d(G\setminus v) + f(G\setminus v) \geq d(G)+f(G)+1,\] where the first inequality follows from \cite[Theorem 3.5]{MKM-depth}. If $\{v\} \subsetneq V(K_{q_{j-1}})\cap V(K_{q_j}),$ then $G \setminus v$ satisfies the hypothesis of the theorem and $\iv(G\setminus v) = \iv(G)-1$. Hence by induction on the number of internal vertices, we get $\depth(S/((x_v,y_v)+J_{G\setminus v})) \geq d(G \setminus v) + f(G \setminus v) + 1.$ If $P'$ is a path in $G$ that does not pass through $v$, then it remains a path in $G\setminus v$. If $P':z=z_1,z_2,\ldots,z_{j-1},v,z_{j+1},\ldots,z_k,z'$ is a path in $G'$, then we can see that $P'':z=z_1, z_2,\ldots, z_{j-1},v',z_{j+1},\ldots,z_k,z'$ is a path in $G\setminus  v$, where $v' \in V(K_{q_{j-1}})\cap V(K_{q_j}) \setminus \{v\}$. Hence $d(G\setminus v)=d(G)$. Therefore, we get $\depth(S/((x_v,y_v)+J_{G\setminus v})) \geq d(G)+f(G)+1$.

Hence, the assertion follows by applying \cite[Proposition 1.2.9]{bh} to the short exact sequence \eqref{ohtani-ses}.
\end{proof}

For a graph $G$, if $d(G)+f(G) = n+2-\kappa(G)$, then $G$ is chordal, \cite{Hibi-Madani}. And,  in that case, they proved that $\depth (S/J_G) = d(G)+f(G)$, (Theorem 2.4, \cite{Hibi-Madani}). If $d(G)+f(G)+1=n+2-\kappa(G)$, then it follows from \Cref{induced-cycle-charac} that $G$ is either chordal or has at most one induced cycle of length $4$ and no cycles of length $\geq 5$. Therefore, we raise the following question:
\begin{question}
If $G$ is a graph containing an induced cycle of length at least $5$, then is $\depth (S/J_G) \geq d(G)+f(G)+1$?
\end{question}
While we believe that this question has an affirmative answer, we are unable to prove this in general. Below, we answer this question affirmatively when $G$ is a unicyclic graph or a quasi-cycle graph.

\begin{proposition}\label{nec-min-depth}
Let $G$ be a unicyclic graph or a quasi-cycle graph such that $\depth(S/J_G)=d(G)+f(G)$. Then either $G$ is  chordal or $G$ has an induced $C_4$.
\end{proposition}
\begin{proof}
Let $G$ be a unicyclic graph or a quasi-cycle graph. Then it follows from \cite[Theorem 3.5]{Rajib} and \cite[Theorem 4.8]{Arv-EJC} that $n\leq \depth(S/J_G) \leq n+1$. If $G$ is a quasi-cycle, then $\kappa(G) \geq 2$. Hence $\depth (S/J_G) \leq n$ so that $\depth (S/J_G) = n = d(G) + f(G)$. Therefore, by \cite[Theorem 2.1]{Hibi-Madani}, $G$ is a chordal graph. Assume that $G$ is unicyclic but not a cycle. Then $\kappa(G)=1$. Now if $\depth (S/J_G) = n$, then $d(G)+f(G)+1=n+2-\kappa(G)$, and hence the assertion follows from \Cref{induced-cycle-charac}. If $\depth (S/J_G) = n+1$, then by \cite[Theorem 2.1]{Hibi-Madani} $G$ is chordal.
\end{proof}

\vskip 3mm \noindent
\textbf{Acknowledgement:} We are thankful to the anonymous referees for going through an initial manuscript and making suggestions which improved the exposition extensively.

\bibliographystyle{plain}
\bibliography{Reference}

\end{document}